\documentclass[12pt,oneside,british,a4wide]{amsart}
\usepackage{amsmath, amssymb,verbatim,appendix}
\usepackage{hyperref}
\usepackage[mathscr]{eucal}
\usepackage{amscd}
\usepackage{amsthm}
\usepackage{stmaryrd}
\usepackage{enumerate}
\usepackage{comment}
\usepackage[latin1]{inputenc} 
\usepackage{tikz}
\usetikzlibrary{shapes,arrows}
\usepackage{cite}
\usepackage{url}

\makeatletter
\newcommand{\dotminus}{\mathbin{\text{\@dotminus}}}

\newcommand{\@dotminus}{%
  \ooalign{\hidewidth\raise1ex\hbox{.}\hidewidth\cr$\m@th-$\cr}%
}
\makeatother

\usepackage{setspace,a4wide,color,xcolor,graphicx}

\newtheorem{theorem}{Theorem}[section]
\newtheorem{lemma}[theorem]{Lemma}
\newtheorem{corollary}[theorem]{Corollary}
\newtheorem{proposition}[theorem]{Proposition}

\newtheorem{problem}[theorem]{Problem}

\newtheorem{claim}[theorem]{Claim}
\newtheorem{fact}[theorem]{Fact}

\theoremstyle{definition}

\newtheorem{definition}[theorem]{Definition}

\def\e{\epsilon}

\newcommand{\VC}{\textnormal{VC}}

\newcommand{\disc}{\textnormal{disc}}

\newcommand{\oct}{\textnormal{oct}}
\newcommand{\dev}{\textnormal{dev}}

\newcommand{\calF}{\mathcal{F}}

\newcommand{\calP}{\mathcal{P}}

\newcommand{\calQ}{\mathcal{Q}}

\newcommand{\calX}{\mathcal{X}}

\def\IP{\operatorname{IP}}

\def\triads{\operatorname{Triads}}

\newenvironment{proofof}[1]{\indent{\scshape Proof of #1}:~~}{\qed}

\begin{document}
\title[]{An improved bound for regular decompositions of $3$-uniform hypergraphs of bounded $\VC_2$-dimension}
\author{C. Terry}\thanks{The author was partially supported by NSF grant DMS-2115518}
\address{Department of Mathematics, The Ohio State University, Columbus, OH 43210, USA}
\email{terry.376@osu.edu}
\maketitle

\begin{abstract}
A regular partition $\calP$ for a $3$-uniform hypergraph $H=(V,E)$ consists of a partition $V=V_1\cup \ldots \cup V_t$ and for each $ij\in {[t]\choose 2}$, a partition $K_2[V_i,V_j]=P_{ij}^1\cup \ldots \cup P_{ij}^{\ell}$, such that certain quasirandomness properties hold.  The \emph{complexity of $\calP$} is the pair $(t,\ell)$.  In this paper we show that if a $3$-uniform hypergraph $H$ has $\VC_2$-dimension at most $k$, then there is such a regular partition $\calP$ for $H$ of complexity $(t,\ell)$, where $\ell$ is bounded by a polynomial in the degree of regularity.  This is a vast improvement on the bound arising from the proof of this regularity lemma in general, in which the bound generated for $\ell$ is of Wowzer type.  This can be seen as a higher arity analogue of the efficient regularity lemmas for graphs and hypergraphs of bounded VC-dimension due to Alon-Fischer-Newman \cite{Alon.2007}, Lov\'{a}sz-Szegedy \cite{Lovasz}, and Fox-Pach-Suk \cite{Fox.2017bfo}.
\end{abstract} 

\section{Introduction}

Szemer\'{e}di's regularity lemma is an important theorem with many applications in extremal combinatorics. The  proof of the regularity lemma, which first appeared in the 70's \cite{Szemeredi}, was well known to produce tower-type bounds in $\e$.  The question of whether this type of bound is necessary was resolved in the late 90's by Gowers' lower bound construction \cite{Gowers.1997}, which showed tower bounds are indeed required (see also \cite{Fox.2014, Moshkovitz.2013, Conlon}).  

Hypergraph regularity was developed in the 2000's by Frankl, Gowers, Kohayakawa, Nagle,  R\"{o}dl, Skokan, Schacht \cite{Frankl.2002, Gowers.20063gk, Gowers.2007, Rodl.2005, Rodl.2004, Nagle.2013}, in order to prove a general counting lemma for hypergraphs.  These types of regularity lemmas are substantially more complicated than prior regularity lemmas.  In particular, a regular partition of a $k$-uniform hypergraph involves a sequence $\calP_1,\ldots, \calP_{k-1}$ where $\calP_i$ is a collection of subsets ${V\choose i}$, such that certain quasi-randomness properties hold for each $\calP_i$ relative to $\calP_1,\ldots, \calP_{i-1}$.  The proofs of these strong regularity lemmas produce Ackerman style bounds for the size of each $\calP_i$.  Given a function $f$, let $f^{(i)}$ denote the $i$-times iterate of $f$.  We then define $Ack_1(x)=2^x$, and for $i>1$, $Ack_i(x)=Ack_{k-1}^{(i)}(x)$.  The proofs of the strong regularity lemma for $k$-uniform hypergraphs produce bounds for the size of each $\calP_i$ of the form $Ack_k$.  It was shown by Moshkovitz and Shapira \cite{Moshkovitz.2019} that this type of bound is indeed necessary for the size of $\calP_1$, which corresponds to the partition of the vertex set.

In the case of $3$-uniform hypergraphs, a decomposition in this sense consists of a partition $\calP_1=\{V_1,\ldots, V_t\}$ of $V$, and a set $\calP_2=\{P_{ij}^{\alpha}:ij\in {[t]\choose 2}, \alpha\in [\ell]\}$, where for each $ij\in {[t]\choose 2}$, $P_{ij}^1\cup \ldots \cup P_{ij}^{\ell}$ is a partition of $K_2[V_i,V_j]$.  The \emph{complexity} of $\calP$ is the pair $(t,\ell)$.  We give a formal statement of the regularity lemma for $3$-graphs here for reference, and refer the reader to Subsection \ref{ss:regularity} for the precise definitions involved.  The version stated below is a refinement of a regularity lemma due to Gowers \cite{Gowers.2007} (for more details see Subsection \ref{ss:regularity}). 

\begin{theorem}[Strong Regularity Lemma for $3$-graphs]\label{thm:reg1} For all $\e_1>0$, and every function $\e_2:\mathbb{N}\rightarrow (0,1]$, there exist positive integers $T_0$, $L_0$, and $n_0$ such that for any $3$-graph $H=(V,E)$ on $n\geq n_0$ vertices, there exists a $\dev_{2,3}(\e_1,\e_2(\ell))$-regular, $(t,\ell,\e_1,\e_2(\ell))$-decomposition $\calP$ for $H$ with $t\leq T_0$ and $\ell \leq L_0$.
\end{theorem}

 In Theorem \ref{thm:reg1}, the parameter $T_0$ is the bound for $t$, the size of the vertex partition, and $L_0$ is the bound for $\ell$, the size of the partition of $K_2[V_i,V_j]$, for each $ij\in {[t]\choose 2}$. The proof of Theorem \ref{thm:reg1} generates a Wowzer (i.e. $Ack_3$) type bound for both $t$ and $\ell$.  Moshkovitz and Shapira showed in \cite{Moshkovitz.2019} that there exist $3$-uniform hyerpgraphs requiring a Wowzer type bound for the size of $t$ in Theorem \ref{thm:reg1}.  Less attention has been paid to the form of the bound $L_0$, and it remains open whether this is necessarily of Wowzer type.  In recent work of the author and Wolf \cite{Terry.2021b}, the partition $\calP_2$ plays a crucial role in the proof of a strong version of Theorem \ref{thm:reg1} in a combinatorially tame setting.   This work suggests that understanding the form of the bound for $\ell$ is also an interesting problem.

In the case of graphs, it was shown that dramatic improvements on the bounds in Szemer\'{e}di's regularity lemma can be obtained under the hypothesis of bounded VC-dimension.  In particular, Lov\'{a}sz and Szegedy  \cite{Lovasz} showed that if a graph $G$ has VC-dimension less than $k$, then it has an $\e$-regular partition of size at most $\e^{-O(k^2)}$.  This was later strengthened to a bound of the form $k\e^{-k}$ by Alon, Fischer, and Newman \cite{Alon.2007}.  Similar bounds were obtained for weak regular partitions of hypergraphs of bounded VC-dimension by Fox, Pach and Suk in \cite{Fox.2017bfo}.  Related results were obtained with weaker polynomial bounds by Chernikov and Starchenko \cite{ChernikovStarchenko}. 

In this paper we prove an analogous theorem in the context of strong regularity for $3$-uniform hypergraphs, where VC-dimension is replaced by a higher arity analogue called $\VC_2$-dimension.

\begin{definition}
Suppose $H=(V,E)$ is a $3$-graph.  The \emph{$\VC_2$-dimension of $H$}, $\VC_2(H)$, is the largest integer $k$ so that there exist vertices $a_1,\ldots, a_k,b_1,\ldots, b_k\in V$ and $c_S\in V$ for each $S\subseteq [k]^2$, such that $a_ib_jc_S\in E$ if and only if $(i,j)\in S$.
\end{definition}

The notion of $\VC_2$-dimension was first introduced \cite{Shelah1} by Shelah, who also studied it in the context of groups \cite{Shelah2}.  It was later shown to have nice model theoretic characterizations by Chernikov-Palacin-Takeuchi \cite{Chernikov.2019},  to have further natural connections to groups and fields by Hempel and Chernikov-Hempel \cite{Hempel.2014, ChernikovHempel}, and to have applications in combinatorics by the author \cite{Terry.2018}.

Using infinitary techniques, Chernikov and Towsner  \cite{Chernikov.2020} proved a strong regularity lemma for $3$-uniform hypergraphs of bounded $\VC_2$-dimension without explicit bounds (in fact they proved results for $k$-uniform hypergraphs of bounded $\VC_{k-1}$-dimension).  Similar results were proved by the author and Wolf \cite{Terry.2021b} in the $3$-uniform case with Wowzer type bounds.  In this paper, we show that $3$-uniform hypergraphs of uniformly bounded $\VC_2$-dimension have regular decompositions with vastly improved bounds on the size of $\ell$; in particular, $\ell$ can be guaranteed to be polynomial in size,  rather than Wowzer.  We include the formal statement of our main theorem below, and refer the reader to the subsequent section for details on the definitions involved.

\begin{theorem}\label{thm:main}
For all $k\geq 1$, there are $\e_1^*>0$ and $\e^*_2:\mathbb{N}\rightarrow (0,1]$ such that the following holds.

Suppose $0<\e_1<\e_1^*$ and $\e_2:\mathbb{N}\rightarrow (0,1]$ satisfies $0<\e_2(x)<\e_2^*(x)$ for all $x\in \mathbb{N}$.  There is $T=T(\e_1,\e_2)$ such that every sufficiently large $3$-graph $H=(V,E)$ has a $\dev_{2,3}(\e_1,\e_2(\ell))$-regular $(t,\ell,\e_1,\e_2(\ell))$-decomposition with $ \ell\leq  \e_1^{-O_k(k)}$, and $t\leq T$.
\end{theorem}

The bound $T$ in Theorem \ref{thm:main} is generated from an application of Theorem \ref{thm:reg1}, and is also of Wowzer type (see Theorem \ref{thm:main2} for a more precise statement regarding this).  The regular partition in Theorem \ref{thm:main} has the additional property that the regular triads have edge densities near $0$ or $1$, which also occurs in the results from \cite{Chernikov.2020, Terry.2021b}.  The ingredients in the proof of Theorem \ref{thm:main} include the improved regularity lemma for $3$-graphs of bounded $\VC_2$-dimension from \cite{Terry.2021b}, a method of producing quotient graphs from regular partitions of $3$-graphs developed in \cite{Terry.2021b}, and ideas from \cite{Fox.2017bfo} for producing weak regular partitions of hypergraphs of bounded VC-dimension.  

The fact that the bound for $\ell$ can be brought all the way down to polynomial in Theorem \ref{thm:main} is somewhat surprising, given that the proof for arbitrary hypergraphs yields a Wowzer bound.  This raises the question of what the correct form of the bound is, in general, for $\ell$.  The author conjectures it is at least a tower function (i.e. $Ack_2$).

It was conjectured in \cite{Chernikov.2020} that the bound for $t$ can also be made sub-Wowzer under the assumption of bounded $\VC_2$-dimension, however, the author has been unable to prove this is the case.  This leaves the following open problem, which we state here.

\begin{problem}
Given a fixed integer $k\geq 1$, are there arbitrarily large $3$-uniform hypergraphs of $\VC_2$-dimension at most $k$, which require Wowzer type bounds for $T_0$ in the Theorem \ref{thm:reg1}?
\end{problem}

\section{Preliminaries}

In this section we cover the requisite preliminaries, including graph and hypergraph regularity (Subection \ref{ss:regularity}), VC and $\VC_2$-dimension (Subsections \ref{ss:vc}, \ref{ss:haussler}, and \ref{ss:vc2}), auxiliary graphs defined from regular decompositions of $3$-graphs (Subsection \ref{ss:aux}), and basic lemmas around regularity and counting (Subsection \ref{ss:counting}).  

\subsection{Notation}

We include here some basic notation needed for the other preliminary sections.  Given a set $V$ and $k\geq 1$, ${V\choose k}=\{X\subseteq V: |X|=k\}$.  A \emph{$k$-uniform hypergraph} is a pair $(V,E)$ where $E\subseteq {V\choose k}$.  For a $k$-uniform hypergraph $G$, $V(G)$ denotes the vertex set of $V$ and $E(G)$ denotes the edge set of $G$.  Throughout the paper, all vertex sets are assumed to be finite.

When $k=2$, we refer to a $k$-uniform hypergraph as simply a \emph{graph}.  When $k=3$, we will refer to a $k$-uniform hypergraph as a \emph{$3$-graph}.  

Given distinct elements $x,y$, we will write $xy$ for the set $\{x,y\}$.  Similarly, for distinct $x,y,z$, we will write $xyz$ for the set $\{x,y,z\}$.  Given sets $X,Y,Z$, we set 
\begin{align*}
K_2[X,Y]&=\{xy: x\in X, y\in Y, x\neq y\}\text{ and }\\
K_3[X,Y,Z]&=\{xyz: x\in X, y\in Y, z\in Z, x\neq y, y\neq z, x\neq z\}.
\end{align*}
If $G=(V,E)$ is a graph and $X,Y\subseteq V$ are disjoint, we let $G[X,Y]$ be the bipartite graph $(X\cup Y, E\cap K_2[X,Y])$.

Given a $k$-uniform hypergraph $G=(V,E)$, $1\leq i<k$, and $e\in {V\choose i}$, set 
$$
N_E(e)=\{e'\in {V\choose k-i}: e\cup e'\in E\}.
$$  

 A \emph{bipartite edge colored graph} is a tuple $G=(A\cup B,E_0,E_1,\ldots, E_i)$, where $i>1$ and $K_2[A,B]=E_0\sqcup E_1\sqcup \ldots \sqcup E_i$.  In this case, given $u\in \{0,1,\ldots, i\}$ and $x\in A\cup B$, we let $N_{E_u}(x)=\{y\in A\cup B: ab\in E_u\}$.
 
 Similarly, a \emph{tripartite edge colored $3$-graph} is a tuple $G=(A\cup B\cup C,E_0,E_1,\ldots, E_i)$, where $i>1$ and $K_3[A,B,C]=E_0\sqcup E_1\sqcup \ldots \sqcup E_i$.  In this case, given $u\in \{0,1,\ldots, i\}$ and $x,y\in V:=A\cup B\cup C$, we let $N_{E_u}(x)=\{uv\in {V\choose 2} : xuv\in E_u\}$ and $N_{E_u}(xy)=\{v\in V: xuv\in E_u\}$.
 
For two functions $f_1,f_2:\mathbb{N}\rightarrow (0,1]$, we write $f_1<f_2$ to denote that $f_1(x)<f_2(x)$ for all $x\in \mathbb{N}$.  For real numbers $r_1,r_2$ and $\e>0$, we write $r_1=r_2\pm \e$ to denote that $r_1\in (r_2-\e,r_2+\e)$.  Given a natural number $n\geq 1$, $[n]=\{1,\ldots, n\}$.  An equipartition of a set $V$ is a partition $V=V_1\cup \ldots \cup V_t$ such that for each $1\leq i,j\leq t$, $||V_i|-|V_j||\leq 1$.  

\subsection{Regularity}\label{ss:regularity}
In this section we define graph regularity, as well as a strong notion of regularity for $3$-graphs. We will state our definitions in terms of the quasi-randomness notion known as ``dev", which is one of three notions of quasirandomness which are now known to be equivalent, the other two being ``oct'' and ``disc.''  For more details on these and the equivalences, we refer the reader to \cite{Nagle.2013}.  

We begin a notion of quasirandomness for graphs.  
\begin{definition}\label{def:dev2}
Suppose $B=(U\cup W, E)$ is a bipartite graph, and $|E|=d_B|U||W|$.  We say $B$ \emph{has $\dev_2(\e,d)$} if $d_B=d\pm \e$ and 
$$
\sum_{u_0,u_1\in U}\sum_{w_0,w_1\in W}\prod_{i\in \{0,1\}}\prod_{j\in \{0,1\}}g(u_i,v_j)\leq \e |U|^2|V|^2,
$$
where $g(u,v)=1-d_B$ if $uv\in E$ and $g(u,v)=-d_B$ if $uv\notin E$. 
\end{definition}

We now define a generalization of Definition \ref{def:dev2} to $3$-graphs due to Gowers \cite{Gowers.20063gk}.  If  $G=(V,E)$ is a graph, let $K_3^{(2)}(G)$ denote the set of triples from $V$ forming a triangle in $G$, i.e.
$$
K_3^{(2)}(G):=\{xyz\in {V\choose 3}: xy, yz, xz\in E\}.
$$
Now given a $3$-graph $H=(V,R)$ on the same vertex set, we say that $G$ \emph{underlies $H$} if $R\subseteq K_3^{(2)}(G)$. 

 \begin{definition}\label{def:regtripdev}
Assume $\e_1,\e_2>0$, $H=(V,E)$ is a $3$-graph, $G=(U\cup W\cup Z,E)$ is a $3$-partite graph underlying $H$, and $|E|=d_3|K^{(2)}_3(G)|$.  We say that \emph{$(H,G)$ has $\dev_{2,3}(\e_1,\e_2)$} if there is $d_2\in (0,1)$ such that $G[U,W]$, $G[U,Z]$, and $G[W,Z]$ each have $\dev_2(\e_2,d_2)$, and 
$$
\sum_{u_0,u_1\in U}\sum_{w_0,w_1\in W}\sum_{z_0,z_1\in Z}\prod_{(i,j,k)\in \{0,1\}^3}h_{H,G}(u_i,w_j,z_k)\leq \e_1 d_2^{12}|U|^2|W|^2|Z|^2,
$$
where $h_{H,G}(x,y,z)=1-d_3$ if $xyz\in E\cap K_3^{(2)}(G)$, $h_{H,G}(x,y,z)=-d_3$ if $xyz\in K_3^{(2)}(G)\setminus E$, and $h_{H,G}(x,y,z)=0$ if $xyz\notin K_3^{(2)}(G)$.
\end{definition}

We now define a $(t,\ell)$-decomposition for a vertex set $V$, which partitions $V$, as well as pairs from $V$.

\begin{definition}\label{def:decomp}
Let $V$ be a vertex set and $t,\ell \in \mathbb{N}^{>0}$. A \emph{$(t,\ell)$-decomposition} $\calP$ for $V$ consists of a partition $\calP_1=\{V_1\cup \ldots \cup V_t\}$ of $V$, and for each $1\leq i\neq j\leq t$, a partition $K_2[V_i,V_j]=P^1_{ij}\cup \ldots \cup P^\ell_{ij}$.  We let $\calP_2=\{P_{ij}^{\alpha}: ij\in {[t]\choose 2}, \alpha\leq \ell\}$.
\end{definition}

A \emph{triad of $\calP$} is a $3$-partite graph of the form $G^{ijk}_{\alpha,\beta,\gamma}:=(V_i\cup V_j\cup V_k, P_{ij}^\alpha\cup P_{ik}^\beta\cup P_{jk}^\gamma)$, for some $ijk\in {[t]\choose 3}$ and $\alpha,\beta,\gamma\leq \ell$.  Let $\triads(\calP)$ denote the set of all triads of $\calP$, and observe that $\{K_3^{(2)}(G): G\in \triads(\calP)\}$ partitions the set of triples $xyz\in {V\choose 3}$ which are in distinct elements of $\calP_1$.  

For a $3$-graph $H=(V,R)$, a decomposition $\calP$ of $V$, and $G\in \triads(\calP)$, define $H|G:=(V(G), R\cap K_3^{(2)}(G))$.  Note that $G$ always underlies $H|G$.  

\begin{definition}
Given a $3$-graph $H=(V,R)$, a decomposition $\calP$ of $V$, and $G\in \triads(\calP)$, we say $G$ \emph{has $\dev_{2,3}(\e_1,\e_2)$ with respect to $H$} if $(H|G,G)$ has $\dev_{2,3}(\e_1,\e_2)$. 
\end{definition}

To define a regular decomposition for a $3$-graph, we need one more notion, namely that of an ``equitable'' decomposition. 

\begin{definition}
We say that $\calP$ is a \emph{$(t,\ell, \e_1,\e_2)$-decomposition} if $\calP_1=\{V_1,\ldots, V_t\}$ is an equipartition and for at least $(1-\e_1){|V|\choose 2}$ many $xy\in {V\choose 2}$, there is some $P_{ij}^{\alpha}\in \calP_2$ containing $xy$ such that $(V_i\cup V_j, P_{ij}^{\alpha})$ has $\dev_{2}(\e_2,1/\ell)$.
\end{definition}

\begin{definition}\label{def:regdec}
Suppose $H=(V,E)$ is a $3$-graph and $\calP$ is an $(t,\ell, \e_1,\e_2)$-decomposition of $V$.  We say that $\calP$ is \emph{$\dev_{2,3}(\e_1,\e_2)$-regular} for $H$ if for all but at most $\e_1 n^3$ many triples $xyz\in {V\choose 3}$, there is some $G\in \triads(\calP)$ with $xyz\in K_3^{(2)}(G)$ such that $G$ has $\dev_{2,3}(\e_1,\e_2)$ with respect to $H$.   
\end{definition}

 We can now restate the regularity lemma for $\dev_{2,3}$-quasirandomness.

\begin{theorem}\label{thm:reg2} For all $\e_1>0$, every function $\e_2:\mathbb{N}\rightarrow (0,1]$, and every $\ell_0,t_0\geq 1$, there exist positive integers $T_0=T_0(\e_1,\e_2,t_0,\ell_0)$ and $L_0=L_0(\e_1,\e_2,t_0,\ell_0)$, such that for every sufficiently large $3$-graph $H=(V,E)$, there exists a $\dev_{2,3}(\e_1,\e_2(\ell))$-regular, $(t,\ell,\e_1,\e_2(\ell))$-decomposition $\calP$ for $H$ with $t_0\leq t\leq T_0$ and $\ell_0\leq \ell \leq L_0$.
\end{theorem}

This theorem was first proved in a slightly different form by Gowers in \cite{Gowers.20063gk}.  In particular, in \cite{Gowers.20063gk}, the partition of the pairs, $\calP_2$, is not required to be equitable as it is in Theorem \ref{thm:reg2}.  Theorem \ref{thm:reg2} as stated appears in \cite{Nagle.2013}, where it is pointed out that the additional equitability requirement can obtained using techniques from \cite{Frankl.2002}.

\vspace{2mm}

\subsection{$\VC$-dimension}\label{ss:vc} In this subsection we give some preliminaries around VC and $\VC_2$-dimension.  We begin by defining VC-dimension.

Given a set $V$, $\calF\subseteq \calP(V)$, and $X\subseteq V$, let $|X\cap \calF|:=\{F\cap X: F\in \calF\}$.  We say that $X$ is \emph{shattered by $\calF$} if $|\calF\cap X|=2^{|X|}$.  The \emph{VC-dimension of $\calF$} is then defined to be the size of the largest subset of $V$ which is shattered by $\calF$.  

For a graph $G=(V,E)$, the \emph{VC-dimension of $G$} is the VC-dimension of the set system $\{N_E(x): x\in V\}\subseteq \calP(V)$.  We now give a simple recharacterization of this.  Given $k\geq 1$, let $A_k=\{a_i: i\in [k]\}$, and $C_{\calP([k])}=\{c_S:S\subseteq [k]\}$.

\begin{definition}
Suppose $k\geq 1$.  Define $U(k)$ to be the bipartite graph $(A_k\cup C_{\calP([k])}, E)$ where $E=\{a_ic_S: i\in S\}$.
\end{definition}

Then it is well known that a graph $G$ has VC-dimension at least $k$ if and only if there is a map $f:V(U(k))\rightarrow V(G)$ so that for all $a\in A_k$ and $c\in C_{\calP([k])}$, $ab\in E(U(k))$ if and only if $f(a)f(b)\in E(G)$. 

\subsection{Encodings}\label{ss:aux}
In this subsection, we define an auxiliary edge-colored graph associated to a regular decomposition of a $3$-graph.   We will then state a result from \cite{Terry.2021b} which shows that encodings of $U(k)$ cannot occur when the auxiliary edge-colored graph arises from a regular decomposition of a $3$-graph with $\VC_2$-dimension less than $k$.

\begin{definition}
Suppose $\e_1,\e_2>0$, $\ell, t\geq 1$, $V$ is a set, and $\calP$ is a $(t,\ell, \e_1,\e_2)$-decomposition for $V$ consisting of $\calP_1=\{V_i: i\in [t]\}$ and $\calP_2=\{P_{ij}^{\alpha}: ij\in {[t]\choose 2}, \alpha\leq \ell\}$.    Define
\begin{align*}
\calP_{cnr}&=\{P_{ij}^\alpha P_{ik}^{\beta}: ijk\in {[t]\choose 3}, \alpha,\beta\leq \ell, \text{ and }P_{ij}^\alpha, P_{ij}^{\beta}\text{ satisfy }\dev_2(\e_2,1/\ell)\}\\
\calP_{edge}&=\{P_{ij}^\alpha \in \calP_2: P_{ij}^\alpha\text{ satisfies }\dev_2(\e_2,1/\ell)\}.
\end{align*}
\end{definition}

In the above, cnr stands for ``corner.''  Observe that for each $P_{ij}^{\alpha}\in \calP_{edge}$ and $P_{uv}^{\beta}P_{uw}^{\gamma}\in \calP_{cnr}$, if $\{v,w\}=\{i,j\}$, then the pair $(P_{ij}^{\alpha}, P_{uv}^{\beta}P_{uw}^{\gamma})$ corresponds to a triad from $\calP$, namely $G_{ijs}^{uvw}$.

\begin{definition}\label{def:reducedP}
Suppose $\e_1,\e_2,\delta>0$, $\ell,t\geq 1$, $H=(V,E)$ is a $3$-graph, and $\calP$ is a $(t,\ell,\e_1,\e_2)$-decomposition for $V$. Define 
\begin{align*}
\mathbf{E}_0(\delta)&=\{P_{ij}^{\alpha}(P_{jk}^{\beta}P_{ik}^{\gamma})\in K_2[\calP_{edge}, \calP_{cnr}]: |E\cap K_3^{(2)}(G_{ijk}^{\alpha\beta\gamma})|\leq \delta|K_3^{(2)}(G_{ijk}^{\alpha\beta\gamma})|\},\\
\mathbf{E}_1(\delta)&=\{P_{ij}^{\alpha}(P_{jk}^{\beta}P_{ik}^{\gamma})\in K_2[\calP_{edge}, \calP_{cnr}]: |E\cap K_3^{(2)}(G_{ijk}^{\alpha\beta\gamma})|\geq (1-\delta)|K_3^{(2)}(G_{ijk}^{\alpha\beta\gamma})|\}\text{ and }\\
\mathbf{E}_2(\delta)&=K_2[\calP_{edge}, \calP_{cnr}]\setminus (\mathbf{E}_1(\delta)\cup \mathbf{E}_0(\delta)).
\end{align*}
\end{definition}

Note that Definition \ref{def:reducedP} gives us a natural bipartite edge colored graph with vertex set $\calP_{edge}\cup \calP_{cnr}$ and edge sets given by $\mathbf{E}_0(\delta), \mathbf{E}_1(\delta),\mathbf{E}_2(\delta)$.  The author and Wolf showed in \cite{Terry.2021b} that these auxiliary edge-colored graphs are useful for understanding $3$-graphs of bounded $\VC_2$-dimension.  To explain why, we require the following notion of an ``encoding.''

\begin{definition}\label{def:encodingR}
Let $\e_1,\e_2,\delta>0$ and $t,\ell\geq 1$.  Suppose $R=(A\cup B,E_R)$ is a bipartite graph,  $H=(V,E)$ is a $3$-graph, and $\calP$ is a $(t,\ell,\e_1,\e_2)$-decomposition of $V$.  An \emph{$(A,B,\delta)$-encoding of $R$ in $(H,\calP)$} consists of a pair of functions $(g,f)$, where $g:A\rightarrow \calP_{cnr}$ and $f:B\rightarrow \calP_{edge}$ are such that the following hold for some $j_0k_0\in {[t]\choose 2}$.
 \begin{enumerate}
 \item $\mathrm{Im}(f)\subseteq \{P_{j_0k_0}^{\alpha}: \alpha\leq \ell\}$, and $Im(g)\subseteq \{P^{\beta}_{ij_0}P^\gamma_{ik_0}: i\in [t], \beta,\gamma\leq \ell\}$, and
 \item For all $a\in A$ and $b\in B$, if $ab\in E_R$, then $g(a)f(b)\in \mathbf{E}_1(\delta)$, and if $ab\notin E_R$, then $g(a)f(b)\in \mathbf{E}_0(\delta)$.
 
 \end{enumerate}
 \end{definition}
 
A \emph{$\delta$-encoding of $U(k)$} will always mean an $(A_k,C_{\calP([k])},\delta)$-encoding of $U(k)$.  In \cite{Terry.2021b}, we proved the following proposition connecting encodings of $U(k)$ and $\VC_2$-dimension (see Proposition 5.6  in \cite{Terry.2021b}).

\begin{proposition}\label{prop:suffvc2}
For all $k\geq 1$ and $\delta\in (0,1/2)$, there are $\e_1>0$ and $\e_2:\mathbb{N}\rightarrow (0,1]$ such that for all $t,\ell\geq 1$, there is $N$ such that the following hold.  Suppose $H=(V,E)$ is a $3$-graph with $|V|\geq N$, and $\calP$ is a $\dev_{2,3}(\e_2(\ell),\e_1)$-regular $(t,\ell, \e_1,\e_2(\ell))$-decomposition of $V$. If there exists a $\delta$-encoding of $U(k)$ in $(H,\calP)$, then $H$ has $k$-$\IP_2$. 

Moreover, there is a constant $C=C(k)$ so that $\e_1=\delta^C$.
\end{proposition}

We remark here that Proposition \ref{prop:suffvc2} is actually proved in \cite{Terry.2021b} for an equivalent notion of quasirandomness called $\disc_{2,3}$, and without the final ``Moreover'' statement regarding the quantitative form for $\e_1$ (see Proposition 5.6 in \cite{Terry.2021b}).  Tracing the bounds in the proof of Proposition 5.6 in \cite{Terry.2021b}, one finds that $\e_1$ has the form $\mu=\mu(\e_1,k)$, where $\mu$ comes from a version of the counting lemma (see Theorem 3.1 in \cite{Terry.2021b}).  An explicit value for this $\mu$ is unclear, as the proof of the counting lemma for $\disc_{2,3}$ passes through its equivalence with $\oct_{2,3}$, and then the counting lemma for $\oct_{2,3}$.  The author has not found proofs of these results in the literature which are explicit in the parameters (see Corollary 2.3 in \cite{Nagle.2013}).  It seems that one could produce such an explicit result from \cite{Nagle.2013} and \cite{Haxell} with some effort, however, we have instead chosen to side-step the issue by working the quasirandomenss notion $\dev$, rather than $\disc$.

In particular, all the ingredients used to prove Proposition 5.6 of \cite{Terry.2021b} have well known analogues for $\dev$.  By running the same arguments as in \cite{Terry.2021b} using $\dev$ rather than $\disc$, one obtains Proposition \ref{prop:suffvc2} as stated.  The additional ``Moreover'' statement about the explicit form for $\e_1$ then arises from the fact that there is a proof of the counting lemma for $\dev_{2,3}$ which is explicit in the parameters (see \cite{Gowers.20063gk}, Theorem 6.8).

\subsection{Haussler's Packing Lemma}\label{ss:haussler}  We will be applying techniques for proving improved regularity lemmas for graphs and hypergraphs of bounded VC-dimension to the edge-colored auxiliary graphs defined in the previous subsection.  In particular, we will use ideas from the proof of Theorem 1.3 in \cite{Fox.2017bfo}. We begin by describing the relevant result from VC-theory, namely Haussler's packing lemma.  

Suppose $V$ is a set and $\calF\subseteq V$.  We say that a subset $\calX\subseteq \calF$ is \emph{$\delta$-separated} if for all distinct $X,X'\in \calX$, $|X\Delta X|>\delta$.  The following packing lemma, due to Haussler, shows that if $\calF$ has bounded VC-dimension, the size a of a $\delta$-separated family cannot be too large \cite{Haussler}.

\begin{theorem}[Haussler's Packing Lemma]\label{thm:haussler}
Suppose $\calF\subseteq \calP(V)$, where $|V|=n$ and $\calF$ has VC-dimension at most $k$.  Then the maximal size of a $\delta$-separated subcollection of $\calF$ is at most $c_1(n/\delta)^k$, for some constant $c_1=c_1(k)$. 
\end{theorem}

We will apply Theorem \ref{thm:haussler} in the setting of edge colored graphs.  This technique is inspired by the  proof of Theorem 1.3 in \cite{Fox.2017bfo}.

Suppose $G=(A\cup B,E_0,E_1,E_2)$ is a bipartite edge colored graph.  We say that $G$ has a $E_0/E_1$-copy of $U(k)$ if there are $v_1,\ldots, v_k\in A$ and for each $S\subseteq [k]$ a vertex $w_S\in B$ such that $i\in S$ implies $v_iw_S\in E_1$ and $i\notin S$ implies $v_iw_S\in E_0$.  Given $a,a'\in A$ and $\delta>0$, write $a\sim_{\delta}a'$ if for each $u\in \{0,1,2\}$, $|N_{E_u}(a)\Delta N_{E_u}(a')|\leq \delta |B|$.  Our main application of Theorem \ref{thm:haussler} is the following lemma.

\begin{lemma}\label{lem:hausslercor}
Suppose $k\geq 1$ and $c_1=c_1(k)$ is as in \ref{thm:haussler}.  Suppose $d\geq 1$ and $\delta,\e>0$ satisfy $\e\leq c_1^{-2}(\delta/8)^{2k+2}$.    Assume $G=(A\cup B, E_0,E_1,E_2)$ is a bipartite edge-colored graph, and assume there is no $E_0/E_1$-copy of $U(k)$ in $G$, and that $|E_2|\leq \e |A||B|$.  

Then there is an integer $m\leq 2c_1(\delta/8)^{-k}$, vertices $x_1,\ldots, x_m\in A$, and a set $U\subseteq A$ with $|U|\leq \sqrt{\e} |A|$,  so that for all $a\in A\setminus U$, $|N_{E_2}(a)|\leq \sqrt{\e}|B|$ and there is some $1\leq i\leq m$ so that $a\sim_{\delta} x_i$. 
\end{lemma}
\begin{proof}
Let $U=\{v\in A: |N_{E_2}(v)|\geq \sqrt{\e}|B|\}$.  Since $|E_2|\leq \e|A||B|$, we know that $|U|\leq \sqrt{\e}|A|$.  Let $A'=A\setminus U$.  Let $m$ be maximal such that there exist $x_1,\ldots, x_m\in A'$, so that $\{N_{E_1}(x_i): i\in [m]\}$ is a $\delta/2$-separated family of sets on $B$.  We show $m\leq 2c_1(\delta/8)^{-k}$.

Suppose towards a contradiction that $m\geq \lceil 2c_1(\delta/8)^{-k}\rceil$. Let $B'=B\setminus (\bigcup_{i=1}^m E_2(x_i))$, and let $\calF:=\{N_{E_1}(x_i)\cap B': i\in [m]\}$.  Notice $|B\setminus B'|\leq m\sqrt{\e}|B|$.  We claim that $\calF$ is $\delta/4$-separated.  Consider $1\leq i\neq j\leq m$.  Then we know that 
\begin{align*}
|N_{E_1}(x_i)\Delta N_{E_1}(x_j)\cap B'|&\geq |N_{E_1}(x_i)\Delta N_{E_1}(x_j)|-m\sqrt{\e}|B|\\
&\geq |B|(\delta/2-m\sqrt{\e})\\
&\geq |B|\delta/4,
\end{align*}
where the last inequality is by our assumptions on $\delta,\e$. By Theorem \ref{thm:haussler}, $\calF$ shatters a set of size $k$.  By construction, for each $1\leq i\leq m$, $B'\setminus N_{E_1}(x_i)\subseteq N_{E_0}(x_i)$.  Consequently, we must have that there exists an $E_0/E_1$-copy of $U(k)$ in $G$, a contradiction.

Thus, $m\leq 2c_1(\delta/8)^{-k}$.    For all $a\in A\setminus U$, we know that $|N_{E_2}(a)|\leq \sqrt{\e}|B|$, and there is some $1\leq i\leq m$ so that $|N_{E_1}(a)\Delta N_{E_1}(x_i)|\leq \delta|B|/2$. We claim that $a\sim_{\delta}x_i$.  We already know that $|N_{E_1}(a)\cap N_{E_1}(x_i)|\leq \delta |B|$.  Since $a,x_i$ are both in $A'$, we have
$$
|N_{E_2}(a)\Delta N_{E_2}(x_i)|\leq |N_{E_2}(a)|+|N_{E_2}(a)|\leq 2 \sqrt{\e}|B|<\delta |B|/2.
$$
Combining these facts, we have that
$$
|N_{E_0}(a)\Delta N_{E_0}(x_i)|\leq |N_{E_2}(a)|+|N_{E_2}(a)|+|N_{E_1}(a)\Delta N_{E_1}(x_i)|\leq \delta |B|.
$$
Thus $a\sim_{\delta}x_i$, as desired.
\end{proof}

\subsection{Tame regularity for $3$-graphs of bounded $\VC_2$-dimension}\label{ss:vc2}
In this subsection we state the tame regularity lemma for $3$-graphs of bounded $\VC_2$-dimension from \cite{Terry.2021b}.

\begin{definition}
Suppose $H=(V,E)$ is a $3$-graph with $|V|=n$ and $\mu>0$. Suppose $t,\ell\geq 1$ and $\calP$ is a $(t,\ell)$-decomposition of $V$.  We say that $\calP$ is \emph{$\mu$-homogeneous with respect to $H$} if at least $(1-\mu){n\choose 3}$ triples $xyz\in {V\choose 3}$ satisfy the following.  There is some $G\in \triads(\calP)$ such that $xyz\in K_3^{(2)}(G)|$ and either $|E\cap K_3^{(2)}(G)|\geq \mu |K_3^{(2)}(G)|$ or $|E\cap K_3^{(2)}(G)\leq (1-\mu)|K_3^{(2)}(G)|$.
\end{definition}

Given a $3$-graph $H=(V,E)$ and a $(t,\ell,\e_1,\e_2)$-decomposition $\calP$ of $V$, we say that $\calP$ is \emph{$\mu$-homogeneous with respect to $H$} if at least $(1-\mu){|V|\choose 3}$ tripes $xyz\in {V\choose 3}$ are in a $\mu$-homogeneous triad of $\calP$.  We have the following theorem from \cite{Terry.2021b}.

\begin{theorem}\label{thm:vc2finite}
For all $k\geq 1$, there are $\e_1^*>0$, $\e_2^*:\mathbb{N}\rightarrow (0,1]$, and a function $f:(0,1]\rightarrow (0,1]$ with $\lim_{x\rightarrow 0}f(x)=0$ such that the following hold.

Suppose $t_0,\ell_0\geq 1$, $0<\e_1<\e_1^*$, and $\e_2:\mathbb{N}\rightarrow (0,1]$ satisfies $\e_2<\e_2^*$. Let $N$, $T$, and $L$ be as in Theorem \ref{thm:reg2} for $\e_1,\e_2,t_0,\ell_0$.  Suppose $H=(V,E)$ is a $3$-graph with $|V|\geq N$ and $\VC_2(H)<k$.  Then there exist $t_0\leq t\leq T$, $\ell_0\leq \ell\leq L$, and a $(t,\ell,\e_1,\e_2(\ell))$-decomposition of $V$ which is $\dev_{2,3}(\e_1,\e_2(\ell))$-regular and $f(\e_1)$-homogeneous with respect to $H$.  

Moreover, $f$ may be taken to have the form $x^{1/D}$ where $D\geq 1$ depends only on $k$.
\end{theorem}

Since the bounds in Theorem \ref{thm:vc2finite} come from Theorem \ref{thm:reg2}, they are of Wowzer type.  We also note that the proof of Theorem \ref{thm:vc2finite} in fact guarantees something slightly stronger, namely that every $\dev_{2,3}(\e_1,\e_2(\ell))$-regular triad of $\calP$ is $f(\e_1)$-homogeneous.

We remark here that Theorem \ref{thm:vc2finite} was proved in \cite{Terry.2021b} for the notion of $\disc_{2,3}$ rather than $\dev_{2,3}$, and without the moreover statement regarding the form of the function $f$ (see Proposition 3.2 in \cite{Terry.2021b}).  Examination of the proof of Proposition 3.2 in \cite{Terry.2021b} shows that the function $f$ depends on $k$ and a version of the counting lemma for $3$-graphs (namely Theorem 3.1 in \cite{Terry.2021b}).  An explicit expression for $f(x)$ in Proposition 3.2 of \cite{Terry.2021b} would thus require a version of the counting lemma for $\disc_{2,3}$ which is explicit in the parameters.  However, one can re-run all the arguments in \cite{Terry.2021b} using the quasi-randomness notion $\dev_{2,3}$ in place of $\disc_{2,3}$ to  obtain Theorem \ref{thm:vc2finite} as stated.   In this case, an explicit expression for $f$ can be obtained using the counting lemma for $\dev_{2,3}$ (see also the discussion following Proposition \ref{prop:suffvc2}).

\subsection{Other Preliminaries}\label{ss:counting}
In this subsection we give several lemmas, most of which are basic facts about regularity and counting.  First, we will use the following version of the triangle counting lemma.  

\begin{proposition}[Counting Lemma]\label{prop:counting}
Suppose $\e,d>0$.  Let $G=(A\cup B\cup C,E)$ be a $3$-partite graph such that each of $G[A,B]$, $G[B,C]$ and $G[A,C]$ has $\dev_2(\e, d)$. Then 
$$
\Big| |K_3^{(2)}(G)|- d^3|A||B||C||\Big|\leq 4\e^{1/4}|A||B||C|.
$$
\end{proposition}

For a proof, see \cite{Gowers.20063gk} Lemma 3.4. The following symmetry lemma was proved in \cite{Terry.2021b} (see Lemma 4.9 there).

\begin{lemma}[Symmetry Lemma]\label{lem:twosticks}
For all $0<\e<1/4$ there is $n$ such that the following holds.  Suppose $G=(U\cup W, E)$ is a bipartite graph, $|U|,|W|\geq n$, and $U'\subseteq U$, $W'\subseteq W$ satisfy $|U'|\geq (1-\e)|U|$ and $|W'|\geq (1-\e)|W|$.  Suppose that for all $u\in U'$, 
$$
\max\{|N(u)\cap W|, |\neg N(u)\cap W|\}\geq (1-\e)|W|,
$$ 
and for all $w\in W'$, 
$$
\max\{|N(w)\cap U|, |\neg N(w)\cap U|\}\geq (1-\e)|U|.
$$ 
Then $|E|/|U||W|\in [0,2\e^{1/2})\cup (1-2\e^{1/2},1]$. 
\end{lemma}

We will use the following immediate corollary of this.

\begin{corollary}\label{cor:twosticks}
For all $0<\e<1/4$ there is $n$ such that the following holds.  Suppose $G=(U\cup W, E)$ is a bipartite graph with $|U|,|W|\geq n$, and $|E|/|U||W|\in (2\e^{1/2},1-2\e^{1/2})$.  Then one of the following hold.
\begin{enumerate}
\item There is $U'\subseteq U$ with $|U'|\geq \e|U|$ so that for all $u\in U$, $\frac{|N_E(u)\cap W|}{|W|}\in (\e,1-\e)$.  
\item There is $W'\subseteq W$ with $|W'|\geq \e|W|$ so that for all $w\in W$, $\frac{|N_E(w)\cap U|}{|U|}\in (\e,1-\e)$.  
\end{enumerate} 
\end{corollary}

We will use a lemma which was orginially proved by Frankl and R\"{o}dl (Lemma 3.8 of \cite{Frankl.2002}) for another notion of quasirandomness for graphs, called $\disc_2$. 

\begin{definition}
Suppose $B=(U\cup W, E)$ is a bipartite graph, and $|E|=d_B|U||W|$.  We say $B$ \emph{has $\disc_2(\e,d)$} if $d_B=d\pm \e$ and for all $U'\subseteq U$ and $W'\subseteq W$,
$$
|E\cap K_2[U',W']|-d|U'||W'||\leq \e |U||W|.
$$
\end{definition}

Gowers proved the following quantitative equivalence between $\disc_2$ and $\dev_2$ (see Theorem 3.1 in \cite{Gowers.20063gk}).

\begin{theorem}\label{thm:equiv}
Suppose $B=(U\cup W, E)$ is a bipartite graph.  If $B$ has $\disc_2(\e,d)$ then it has $\dev_2(\e,d)$.  If $B$ has $\dev_2(\e,d)$, then it has $\disc_2(\e^{1/4},d)$. 
\end{theorem}

Combining Theorem \ref{thm:equiv} with Lemma 3.8 in \cite{Frankl.2002}, we obtain the following.

\begin{lemma}\label{lem:3.8}
For all $\e>0$, $\rho\geq 2\e$, $0<p<\rho/2$, and $\delta>0$, there is $m_0=m_0(\e, \rho, \delta)$ such that the following holds.  Suppose $|U|=|V|=m\geq m_0$, and $G=(U\cup V, E)$ is a bipartite graph  satisfying $\dev_2(\e)$ with density $\rho$.  Then if $\ell=[1/p]$ and $\e\geq 10(1/\ell m)^{1/5}$, there is a partition $E=E_0\cup E_1\cup \ldots \cup E_\ell$,  such that 
 \begin{enumerate}
 \item For each $1\leq i\leq \ell $, $(U\cup V, E_i)$ has $\dev_2(\e^{1/4})$ with density $\rho p(1\pm \delta)$, and
 \item $|E_0|\leq \rho p(1+\delta)m^2$.
 \end{enumerate}
 Further, if $1/p\in \mathbb{Z}$, then $E_0=\emptyset$.
\end{lemma} 

We will also use the following fact, which can be obtained from Fact 2.3 in \cite{Terry.2021b} along with Theorem \ref{thm:equiv}.

\begin{fact}\label{fact:adding}
Suppose $E_1$ and $E_2$ are disjoint subsets of $K_2[U,V]$.  If $(U\cup V, E_1)$ has $\dev_2(\e_1,d_1)$, and $(U\cup V, E_2)$ has $\dev_2(\e_2,d_2)$, then $(U\cup V, E_1\cup E_2)$ has $\dev_2(\e_1^{1/4}+\e_2^{1/4}, d_2+d_1)$.  
\end{fact}

Finally, we will use the fact that triads with density near $0$ or $1$ are quasirandom.  For completeness, we include a proof of this in the appendix.

\begin{proposition}\label{prop:homimpliesrandome}
For all $0<\e<1/2$, $d_2>0$, and $0<\delta\leq (d_2/2)^{48}$, there is $N$ such that the following holds.  Suppose $H=(V_1\cup V_2\cup V_3, R)$ is a $3$-partite $3$-graph on $n\geq N$ vertices, and for each $i,j\in [3]$, $||V_i|-|V_j||\leq \delta|V_i|$.  Suppose $G=(V_1\cup V_2\cup V_3, E)$ is a $3$-partite graph, where for each $1\leq i< j\leq 3$, $G[V_i,V_j]$ has $\dev_2(\delta,d_2)$, and assume
$$
|R\cap K_3^{(2)}(G))|\leq \e|K_3^{(2)}(G)|.
$$
Then $(H|G,G)$ has $\dev_{2,3}(\delta,6\e)$.
\end{proposition}

\section{Proof of Main Theorem}

We first give a more precise statement of our main theorem.

\begin{theorem}\label{thm:main2}
For all $k\geq 1$, there are polynomials $p_1(x), p_2(x,y), p_3(x)$, a constant $\e_1^*>0$, and a function $\e^*_2:\mathbb{N}\rightarrow (0,1]$ such that the following holds, where $T_0(x,y,z,w)$ is as in Theorem \ref{thm:reg2}.

For all $0<\e_1<\e_1^*$ and $\e_2:\mathbb{N}\rightarrow (0,1]$ satisfying $\e_2<\e_2^*$, there is $L\leq \e_1^{-O_k(k)}$ such that the following holds for $T=T_0(p_1(\e_1), \e_2\circ q_2, p_3(\e_1^{-1}), 1)$, where $q_2(y)=p_2(\e_1,y)$.

Every sufficiently large $3$-graph $H=(V,E)$ with $\VC_2(H)<k$ has a $\dev_{2,3}(\e_1,\e_2(\ell))$-regular $(t,\ell,\e_1,\e_2(\ell))$-decomposition with $ \ell\leq  L$ and $t\leq T$.
\end{theorem}

We now give a few remarks regarding the bounds. As can be seen above, the bound $T$ in Theorem \ref{thm:main2} is obtained by composing the bound $T_0$ from Theorem \ref{thm:reg2} with several polynomial functions.  This does not change the fundamental shape of the bound in terms of the Ackerman heirarchy, and thus the bound for $t$ in Theorem \ref{thm:main2} remains a Wowzer type function.  On the other hand, we see that the bound for $\ell$ becomes polynomial in $\e_1^{-1}$.  

The polynomial $p_3$ in Theorem \ref{thm:main2} depends on the $f$ in Theorem \ref{thm:vc2finite}, which in turn depends on the hypergraph counting lemma for $\dev_{2,3}$.  One could therefore obtain a quantitative version of Theorem \ref{thm:main2} for the equivalent quasirandomness notions of $\disc_{2,3}$ and $\oct_{2,3}$ using the same arguments, given a quantitative version of their respective counting lemmas. 

The general strategy for the proof of Theorem \ref{thm:main2} is as follows.  Given a large $3$-graph $H$ of $\VC_2$-dimension less than $k$, we will first apply Theorem \ref{thm:vc2finite} to obtain a homogeneous, regular partition $\calP$ for $H$.  We will then consider the auxiliary edge-colored graphs associated to $\calP$, as described in Subsection \ref{ss:aux}.  These will contain no copies of $U(k)$ by Proposition \ref{prop:suffvc2}, allowing us to apply Lemma \ref{lem:hausslercor}.  This will yield decompositions for the auxiliary edge colored graphs, which we will eventually use to define a new decomposition $\calQ$ for $H$ which is still regular and homogeneous, but which has a polynomial bound for the parameter $\ell$.  This last part requires the most work, as well as most of the lemmas from Subsection \ref{ss:counting}.

We have not sought to optimize constants which do not effect the overall form of the bounds involved.

\vspace{2mm}
\begin{proofof}{Theorem \ref{thm:main}}
Fix $k\geq 1$ and let $c_1=c_1(k)$ be as in Theorem \ref{thm:haussler}.  Let $\rho_1>$, $\rho_2:\mathbb{N}\rightarrow (0,1]$, and $f$ be as in Theorem \ref{thm:vc2finite} for $k$, and let $D=D(k)$ be so that $f(x)=x^{1/D}$ (see Theorem \ref{thm:vc2finite}).  Let $\mu_1>0$ and $\mu_2:\mathbb{N}\rightarrow (0,1]$ be as in Proposition \ref{prop:suffvc2} for $k$ and $1/4$.  Set $\e_1^*=\min\{\mu_1,\rho_1,(1/4)^D\}$ and define $\e^*_2:\mathbb{N}\rightarrow (0,1]$ by setting $\e^*_2(x)=\min\{\mu_2(x),\rho_2(x), (1/2x)^{48}\}$, for each $x\in \mathbb{N}$.  

Suppose $0<\e_1<\e_1^*$ and $\e_2:\mathbb{N}\rightarrow (0,1]$ satisfies $\e_2<\e_2^*$.  We now choose a series of new constants.  Set $\tau_1=\e_1^{4D}$ and note $\tau_1<f(\e_1)$. Set $\delta=\tau_1^{400}/1000$, $\e_1'=(\delta/8c_1)^{2k+1000}$, $m=\lceil 2c_1(\delta/8)^{-2k-2}\rceil$, and $\e_1''=(\e_1')^2/1000$.  Define $\e_2',\e_2'':\mathbb{N}\rightarrow (0,1]$ by setting, for each $x\in \mathbb{N}$, $\e_2'(x)=\e_1''\e_2(x)\e_2(2^4\delta^{-8k-10})$ and $\e_2''(x)=\e_2(\delta^{-4}m^4)\e_2'(x)^5/4$.  Note there are polynomials $p_1(x)$, $p_2(x,y)$ depending only on $k$ such that $\e_1''=p_1(\e_1)$  and $\e_2''(x)=p_2(\e_1,x)$.  To aid the reader in keeping track of the constants, we point out that the following inequalities hold.
$$
\e_1''<\e_1'<\delta<\tau_1<\e_1<\e_1^* \text{ and }\e_2''<\e_2'<\e_2<\e_2^*.
$$

Choose $t_0$ sufficiently large so that $t^3/6\geq (1-\e_1''){t\choose 3}$, $\frac{(1-3\e_1'')t^3}{12}\geq (1-\e_1'){t\choose 3}$, and
$$
{t\choose 3}(1-6(\e_1')^{1/4}-(\e_1')^{3/8})\geq {t\choose 3}(1-(\e_1'))^{1/8}.
$$
  Note there is some polynomial $p(x)$ depending only on $k$ so that we can take $t_0=p(\e_1^{-1})$. Finally, choose $T_1$, $L_1$, and $N_1$  as in Theorem \ref{thm:reg2}  for $\e''_1$, $\e''_2$, $t_0$ and $\ell_0=1$.  
  
Set $L=\lceil \delta^{-4}m^4\rceil$, $T=T_1$, choose $N$ sufficiently large compared to all the previously chosen constants.   Notice that $L=O_k(\e_1^{-O_k(1)})$ and $T=T_0(p_1(\e_1), \e_2\circ q_2, p(\e_1^{-1}), 1)$, where $T_0(x,y,z,w)$ is as in Theorem \ref{thm:reg2} and $q_2(y)=p_2(\e_1,y)$.

Suppose $H=(V,E)$ is a $3$-graph with $|V|\geq N$ satisfying $\VC_2(H)<k$.  Theorem \ref{thm:vc2finite} implies there exist $1\leq \ell\leq L_1$, $t_0\leq t\leq T_1$, and $\calP_1$ a  $(t,\ell,\e''_1,\e''_2(\ell))$-decomposition of $V$ which is $\dev_{2,3}(\e''_1,\e''_2(\ell_1))$-regular and $f(\e''_1)$-homogeneous with respect to $H$.  Say $\calP_1=\{V_1,\ldots, V_t\}$ and $\calP_2=\{P_{ij}^\alpha: ij\in {[t]\choose 2}, \alpha \in [\ell]\}$.  Note that $f(\e_1'')=(\e_1'')^{1/D}<1/4$.  Recall that as mentioned after Theorem \ref{thm:vc2finite}, we may assume that all $\dev_{2,3}(\e_1'',\e_2''(\ell))$-regular triads of $\calP$ are $f(\e_1'')$-homogeneous with respect to $H$.

Given $ij\in {[t]\choose 2}$ and $\alpha\in [\ell]$, let $G_{ij}^{\alpha}=(V_i\cup V_j, P_{ij}^{\alpha})$.  Given $ijs\in {[t]\choose 3}$ and $1\leq \alpha,\beta,\gamma\leq \ell$, set $G_{ijs}^{\alpha,\beta,\gamma}=(V_i\cup V_j\cup V_s, P_{ij}^\alpha\cup P_{js}^\beta\cup P_{is}^\gamma)$ and $H_{ijs}^{\alpha,\beta,\gamma}=(V_i\cup V_j\cup V_s, E\cap K_3^{(2)}(G_{ijs}^{\alpha,\beta,\gamma}))$.  

We will use throughout that since $\e_2''(x)\leq \e_2'(x)^5/4$, Proposition \ref{prop:counting} implies that for all $ijs\in {[t]\choose 3}$ and $\alpha,\beta,\gamma\in [\ell]$, 
\begin{align}\label{align:1}
|K_3^{(2)}(G_{ijs}^{\alpha,\beta,\gamma})|=(1\pm \e'_2(\ell))(n/\ell t)^3.
\end{align}
We will use $\calP$ to construct a different decomposition of $V$, which will be called $\calQ$, so that $\calQ_1=\calP_1$ but $\calQ_2\neq \calP_2$. Set 
\begin{align*}
\mathbf{F}_{err}&=\{G_{ijs}^{\alpha,\beta,\gamma} \in \triads(\calP):(H_{ijs}^{\alpha,\beta,\gamma},G_{ijs}^{\alpha,\beta,\gamma}) \text{ fails }\disc_{3}(\e_1'',\e''_2(\ell))\}\\
\mathbf{F}_1&=\{G_{ijs}^{\alpha,\beta,\gamma} \in \triads(\calP)\setminus \Gamma_{err}: d_{ijs}^{\alpha,\beta,\gamma}\geq 1-f(\e_1'')\},\text{ and }\\
\mathbf{F}_0&=\{G_{ijs}^{\alpha,\beta,\gamma} \in \triads(\calP)\setminus \Gamma_{err}:  d_{ijs}^{\alpha,\beta,\gamma}\leq f(\e_1'')\}.
\end{align*}
By assumption, $\triads(\calP)=\mathbf{F}_{err}\sqcup \mathbf{F}_1\sqcup \mathbf{F}_0$, and at most $\e_1''n^3$ triples $xyz\in {V\choose 3}$ are in $K_3^{(2)}(G)$ for some $G\in \mathbf{F}_{err}$.  By (\ref{align:1}), this implies 
$$
|\triads(\calP)\setminus \mathbf{F}_{err}|\geq \Big({n\choose 3}-\e_1''n^3\Big)/((n^3/t^3\ell^3)(1-\e_2'(\ell))\geq {t\choose 3}\ell^3(1-\e_1'),
$$
where the last inequality uses that $t\geq t_0$ and $n$ is large.  Thus, $|\mathbf{F}_{err}|\leq  \e_1' t^3\ell^3$.  Let
\begin{align*}
\Psi&=\{V_iV_j: |\{G_{ijs}^{\alpha\beta\gamma}\in \mathbf{F}_{err}\text{ some }s\in [t]\text{ and }\alpha,\beta,\gamma\in [\ell]\}| \geq (\e_1')^{3/4}\ell^3 t \}.
\end{align*}
Since $|\mathbf{F}_{err}|\leq \e_1' t^3\ell^3$, we have that $|\Psi|\leq (\e_1')^{1/4} t^2$.  Given $ij\in {[t]\choose 2}$, let $\ell_{ij}$ be the number of $\alpha\in [\ell]$ such that $G_{ij}^{\alpha}$ has $\dev_2(\e_2''(\ell), 1/\ell)$.  After relabeling, we may assume $G_{ij}^1,\ldots, G_{ij}^{\ell_{ij}}$ each have $\dev_2(\e_2''(\ell),1/\ell)$.  We claim that for $V_iV_j\notin \Psi$, $\ell_{ij}\geq (1-2(\e_1')^{3/4})\ell$.  Indeed, given $V_iV_j\notin \Psi$, if it were the case that $\ell_{ij}<(1-2(\e_1')^{3/4})\ell$, then we would have that 
$$
|\{G_{ijs}^{\alpha\beta\gamma}\in \mathbf{F}_{err}\text{ some }s\in [t]\text{ and }\alpha,\beta,\gamma\in [\ell]\}| \geq (t-2)\ell^2(\ell-\ell_{ij})>2(\e_1')^{3/4}(t-2)\ell^3 \geq (\e_1')^{3/4}\ell^3 t,
$$
contradicting that $V_iV_j\notin \Psi$.  Thus we have that for all $V_iV_j\notin \Psi$, $\ell_{ij}\geq (1-2(\e_1')^{3/4})\ell$.

For each $V_iV_j\notin \Psi$, define $H_{ij}$ to be the edge colored graph $(U_{ij}\cup W_{ij},E^0_{ij},E_{ij}^1,E_{ij}^2)$, where 
\begin{align*}
W_{ij}&=\{P_{ij}^\alpha:\alpha\leq \ell_{ij} \}\text{ and }U_{ij}=\{P_{is}^\beta P_{js}^\gamma: s\in [t]\setminus\{i,j\}, \beta\leq \ell_{is},\gamma\leq \ell_{js}\},\text{ and }\\
E^1_{ij}&=\{P_{ij}^\alpha (P_{is}^\beta P_{js}^\gamma)\in K_2[W_{ij},U_{ij}]: G_{ijk}^{\alpha,\beta,\gamma}\in \mathbf{F}_1\}\\
E^0_{ij}&=\{P_{ij}^\alpha (P_{is}^\beta P_{js}^\gamma)\in K_2[W_{ij},U_{ij}]: G_{ijk}^{\alpha,\beta,\gamma}\in \mathbf{F}_0\}\\
E^2_{ij}&=\{P_{ij}^\alpha (P_{is}^\beta P_{js}^\gamma)\in K_2[W_{ij},U_{ij}]: G_{ijk}^{\alpha,\beta,\gamma}\in \mathbf{F}_{err}\}.
\end{align*}
By Proposition \ref{prop:suffvc2}, and since $f(\e_1'')<1/4$, $H_{ij}$ contains no $E_{ij}^1/E_{ij}^0$ copy of $U(k)$, and since $V_iV_j\notin \Psi$, $|E_{ij}^2|\leq (\e_1')^{3/4}\ell^3t$.  We will later need the following size estimates for $W_{ij}$ and $U_{ij}$.  By above, $|W_{ij}|=\ell_{ij}\geq (1-2(\e_1')^{3/4})\ell$.  We claim that $|U_{ij}|\geq (1-2(\e_1')^{3/4})\ell^2t$.  Indeed, observe that  $|U_{ij}|=\sum_{s\in [t]\setminus \{i,j\}}\ell_{is}\ell_{js}$ and 
\begin{align*}
|\{G_{ijs}^{\alpha\beta\gamma}\in \mathbf{F}_{err}\text{ some }s\in [t]\text{ and }\alpha,\beta,\gamma\in [\ell]\}|&\geq\sum_{s\in [t]\setminus \{i,j\}}\ell^2(\ell-\ell_{is})+\ell_{is}\ell(\ell-\ell_{js})\\
&= \sum_{s\in [t]\setminus \{i,j\}}\ell^3-\ell \ell_{is}\ell_{js}=(t-2)\ell^3-\ell |U_{ij}|.
\end{align*}
Since $V_iV_j\notin \Psi$ this shows that $(\e_1')^{3/4}\ell^3 t\geq (t-2)\ell^3-\ell |U_{ij}|$.  Rearranging, this yields that 
$$
|U_{ij}|\geq (t-2)\ell^2-(\e_1')^{3/4}\ell^2 t\geq t\ell^2(1-2(\e_1')^{3/4}),
$$
where the last inequality is because $t\geq t_0$.

Given $v,v'\in W_{ij}$, write $v\sim v'\in W_{ij}$ if for each $w\in \{0,1,2\}$, $|E_{ij}^w(v)\Delta E_{ij}^w(v')|\leq \delta |U_{ij}|$.  By Lemma \ref{lem:hausslercor}, there is $W_{ij}^0\subseteq W_{ij}$ of size at most $(\e_1')^{3/8}|W_{ij}|$, an integer $m_{ij}\leq m$, and $x_{ij}^1,\ldots, x_{ij}^{m_{ij}}\in W_{ij}$ so that for all $v\in W_{ij}\setminus W_{ij}^0$, there is $1\leq \alpha\leq m_{ij}$ so that $v\sim x_{ij}^{\alpha}$, and further, $|N_{E_{ij}^2}(v)|\leq (\e_1')^{3/8}|U_{ij}|$.  For each $1\leq u\leq m_{ij}$, let 
$$
W_{ij}^u=\{v\in W_{ij}\setminus W_{ij}^0: v\sim x_{ij}^u\text{ and for all $1\leq u'<u$}, v\nsim x_{ij}^{u'}\}.
$$
Note $W_{ij}^1\cup \ldots \cup W_{ij}^{m_{ij}}$ is a partition of $W_{ij}\setminus W_{ij}^0$.

We now define a series of sets to help us zero in on certain well behaved sets of triples.  First, define
\begin{align*}
\Omega_0&=\{ijs\in {[t]\choose 3}: V_iV_j, V_jV_s, V_iV_s\notin \Psi\}\text{ and }\\
\Omega&=\{W_{ij}^uW_{is}^vW_{js}^w: ijs\in \Omega_0, 1\leq u\leq m_{ij}, 1\leq v\leq m_{is}, 1\leq w\leq m_{js}\}.
\end{align*}
Since $|\Psi|\leq (\e_1')^{1/4} t^2$, $|\Omega_0|\geq {t\choose 3}-|\Psi|t\geq (1-6(\e_1')^{1/4}){t\choose 3}$.  Let $Y_0=\bigcup_{W_{ij}^uW_{is}^vW_{js}^w\in \Omega}K_3[W_{ij}^u,W_{is}^v,W_{js}^w]$.  We have that for all $ijs\in \Omega_0$, $|W_{ij}^0|,|W_{is}^0|,|W_{js}^0|\leq (\e_1')^{3/8}\ell$, and therefore $|Y_0|$ is at least the following.
\begin{align*}
|Y_0|\geq {t\choose 3}\ell^3-\ell^3|{[t]\choose 3}\setminus \Omega_0|-|\Omega_0|(\e_1')^{3/8}\ell^3&\geq {t\choose 3}\ell^3-6(\e_1')^{1/4}{t\choose 3}\ell^3-{t\choose 3}(\e_1')^{3/8}\ell^3\\
&\geq  {t\choose 3}\ell^3(1-(\e'_1)^{1/8}),
\end{align*}
where the last inequality is since $t\geq t_0$.

Given $ij\notin \Psi$, let us call $W_{ij}^u$ \emph{non-trivial} if it has size at least $\delta^{1/2}\ell /m_{ij}$.  Define
$$
\Omega_1=\{W_{ij}^uW_{js}^vW_{is}^w\in \Omega: \text{ each of }W_{ij}^u, W_{js}^v, W_{is}^w\text{ are non-trivial}\},
$$
and set $Y_1=\bigcup_{W_{ij}^uW_{js}^vW_{is}^w\in \Omega_1}K_3[W_{ij}^uW_{js}^vW_{is}^w]$.  Then we have that
$$
|Y_1|\geq |Y_0|-t\ell^2\sum_{ij\in {[t]\choose 2}}\sum_{\{u\in [m_{ij}]: W^u_{ij}\text{ trivial}\} }\delta^{1/2}(\ell/m_{ij})\geq |Y_0|-t\ell^2(t^2\delta^{1/2} \ell)= |Y_0|-\delta^{1/2} t^3\ell^3.
$$
Define
\begin{align*}
R_1&=\{P_{ij}^\alpha P_{is}^\beta P_{js}^\gamma: G_{ijk}^{\alpha,\beta,\gamma}\in \mathbf{F}_1\}\\
R_0&=\{P_{ij}^\alpha P_{is}^\beta P_{js}^\gamma: G_{ijk}^{\alpha,\beta,\gamma}\in \mathbf{F}_0\}\\
R_2&=\{P_{ij}^\alpha P_{is}^\beta P_{js}^\gamma: G_{ijk}^{\alpha,\beta,\gamma}\in \mathbf{F}_{err}\}
\end{align*}
Note that $(\calP_2\cup \calP_2\cup \calP_2, R_0,R_1,R_2)$ is a $3$-partite edge colored $3$-graph, and $|R_2|\leq \e_1't^3\ell^3$.  Now set 
$$
\Omega_2=\{W_{ij}^uW_{js}^vW_{is}^w\in \Omega_1: |R_2\cap K_3[W_{ij}^u,W_{is}^v,W_{js}^w]|\leq \sqrt{\e_1'} |W_{ij}^u||W_{is}^v||W_{js}^w|\}.
$$
and $Y_2=\bigcup_{W_{ij}^uW_{js}^vW_{is}^w\in \Omega_2}K_3[W_{ij}^uW_{js}^vW_{is}^w]$. Note that 
\begin{align*}
|R_2|&\geq \sum_{W_{ij}^uW_{js}^vW_{is}^w\in \Omega_1\setminus \Omega_2}\sqrt{\e_1'}|W_{ij}^u||W_{is}^v||W_{js}^w|\geq  \sqrt{\e_1'} \sum_{W_{ij}^uW_{js}^vW_{is}^w\in \Omega_1\setminus \Omega_2}|W_{ij}^u||W_{is}^v||W_{js}^w|.
\end{align*}
Therefore,
$$
\sum_{W_{ij}^uW_{js}^vW_{is}^w\in \Omega_1\setminus \Omega_2}|W_{ij}^u||W_{is}^v||W_{js}^w|\leq\sqrt{\e_1'}^{-1}|R_2|<\sqrt{\e_1'}^{-1}\e_1't^3\ell^3\leq \sqrt{\e_1'}t^3\ell^3.
$$
This implies that $|Y_2|\geq |Y_1|-\sqrt{\e_1'}t^3\ell^3$.

Given $ijs\in \Omega_0$, let us call a triple  $P_{ij}^{\alpha}P_{is}^{\beta}P_{js}^{\gamma}$ \emph{troublesome} if one of the following hold.
\begin{itemize}
\item For some $u\in [m_{ij}]$, $P_{ij}^{\alpha}\in W_{ij}^u$, and there are $\sigma_1\neq \sigma_2\in \{0,1,2\}$ such that $P_{is}^{\beta}P_{js}^{\gamma}P_{ij}^{\alpha}\in R^{\sigma_1}$ and $P_{is}^{\beta}P_{js}^{\gamma}x_{ij}^u\in R^{\sigma_1}$.  
\item For some $w\in [m_{js}]$, $P_{js}^{\gamma}\in W_{js}^w$, and there are $\sigma_1\neq \sigma_2\in \{0,1,2\}$ such that $P_{is}^{\beta}P_{ij}^{\alpha}P_{js}^{\gamma}\in R^{\sigma_1}$ and $P_{is}^{\beta}P_{ij}^{\alpha}x_{js}^w\in R^{\sigma_2}$.  
\item For some $v\in [m_{is}]$, $P_{is}^{\beta}\in W_{is}^v$, and there are $\sigma_1\neq \sigma_2\in \{0,1,2\}$ such that $P_{ij}^{\alpha}P_{js}^{\gamma}P_{is}^{\beta}\in R^{\sigma_1}$ and $P_{ij}^{\alpha}P_{js}^{\gamma}x_{is}^v\in R^{\sigma_2}$.  
\end{itemize}

Let $Tr$ be the set of troublesome triples.  Define
$$
\Omega_3=\{W_{ij}^uW_{js}^vW_{is}^w\in \Omega_2: |K_3[W_{ij}^uW_{js}^vW_{is}^w]\cap Tr|\leq \delta^{1/4}|W_{ij}^u||W_{js}^v||W_{is}^w|\},
$$
and set $Y_3=\bigcup_{W_{ij}^uW_{js}^vW_{is}^w\in \Omega_3}K_3[W_{ij}^uW_{js}^vW_{is}^w]$.  We claim $|Y_3|\geq {t\choose 3}\ell^3(1-2\delta^{1/2})$.  

Given $V_iV_j\notin \Psi$, $1\leq u\leq m_{ij}$, and $P_{ij}^{\alpha}\in W_{ij}^u$, we know that $P_{ij}^{\alpha}\sim x_{ij}^{\alpha}$, and therefore
\begin{align*}
|\{P_{is}^{\beta}P_{js}^{\gamma}: s\in [t]\setminus \{i,j\}, \beta,\gamma\leq \ell, P_{is}^{\beta}P_{js}^{\gamma}P_{ij}^{\alpha}\in Tr\}|&\leq (\ell^2(t-2)-|U_{ij}|)\\
&+\sum_{x=0}^2|N_{E_{ij}^x}(P_{ij}^{\alpha})\Delta N_{E_{ij}^x}(x_{ij}^u)|\\
&\leq 2(\e_1')^{3/4}\ell^2 t +3\delta t \ell^2\\
&\leq 4\delta t\ell^2.
\end{align*}
Thus, $|Tr|\leq 4\delta t\ell^2(\sum_{V_iV_j\notin \Psi, u\in [m_{ij}]}|W_{ij}^u|)\leq 4\delta t\ell^2 (t^2\ell)=4\delta t^3\ell^3$.  Therefore 
\begin{align*}
4\delta t^3\ell^3\geq |Tr|&\geq \sum_{W_{ij}^uW_{js}^vW_{is}^w\in \Omega_2\setminus \Omega_3}\delta^{1/4} |W_{ij}^u||W_{js}^v||W_{is}^w|=\delta^{1/4} |\bigcup_{W_{ij}^uW_{js}^vW_{is}^w\in \Omega_2\setminus \Omega_3}K_3[W_{ij}^uW_{js}^vW_{is}^w]|.
\end{align*}
Rearranging, this yields that
\begin{align*}
|\bigcup_{W_{ij}^uW_{js}^vW_{is}^w\in \Omega_2\setminus \Omega_3}K_3[W_{ij}^uW_{js}^vW_{is}^w]|\leq \delta^{-1/4}4\delta t^3\ell^3=4\delta^{3/4}t^3\ell^3.
\end{align*}
Thus 
\begin{align*}
|Y_3|\geq |Y_2|-\delta^{3/4}t^3\ell^3&\geq |Y_1|-\sqrt{\e_1'}t^3\ell^3-4\delta^{3/4}t^3\ell^3\\
&\geq |Y_0|-\delta^{1/2} t^3\ell^3-\sqrt{\e_1'}t^3\ell^3-4\delta^{3/4}t^3\ell^3\\
&\geq {t\choose 3}\ell^3(1-7(\e'_1)^{1/8})-\delta^{1/2} t^3\ell^3-\sqrt{\e_1'}t^3\ell^3-4\delta^{3/4}t^3\ell^3\\
&\geq {t\choose 3}\ell^3(1-2\delta^{1/2}).
\end{align*}
Therefore, using (\ref{align:1}), we have
\begin{align*}
|\bigcup_{P_{ij}^{\alpha}P_{is}^{\beta}P_{js}^{\gamma}\in Y_3} K_3^{(2)}(G_{ijs}^{\alpha,\beta,\gamma})|&\geq {t\choose 3}\ell^3(1-2\delta^{1/2})\Big(\frac{n^3}{t^3\ell^3}(1-\e_2'(\ell)\Big)\geq {n\choose 3}(1-3\delta^{1/2}),
\end{align*}
where the last inequality is because $n$ is large.

Our next goal is to prove Claim \ref{cl:hom} below, which says that each for $W_{ij}^uW_{is}^vW_{js}^w\in \Omega_3$, $K_3[W_{ij}^u,W_{is}^v,W_{js}^w]$ is either mostly contained in $R_1$ or mostly contained in $R_0$.  For the proof of this claim, we will require the following notation.  Given $ijs\in {[t]\choose 3}$, $\alpha,\alpha'\leq \ell$, $1\leq v\leq m_{is}$, and $1\leq w\leq m_{js}$, we write $P_{ij}^{\alpha}\sim_{js,vw} P_{ij}^{\alpha'}$ if $P_{ij}^{\alpha},P_{ij}^{\alpha'}\in W_{ij}^u$ for some $1\leq u\leq m_{ij}$, and 
\begin{align*}
|\{(P_{is}^{\beta},P_{js}^{\gamma})\in W_{is}^v\times W_{js}^w: \text{for some }\sigma_1\neq \sigma_2\in \{0,1,2\},  P_{is}^{\beta}P_{js}^{\gamma}P_{ij}^{\alpha}\in R^{\sigma_1}&\text{ and }P_{is}^{\beta}P_{js}^{\gamma}P_{ij}^{\alpha'}\in R^{\sigma_2}\}|\\
&\leq \delta^{1/8}|W_{is}^v||W_{js}^w|.
\end{align*}

\begin{claim}\label{cl:hom}
For any $W_{ij}^uW_{is}^vW_{js}^w\in \Omega_3$, there is $\sigma\in \{0,1\}$ such that 
$$
\frac{|R_{\sigma}\cap K_3[W_{ij}^u,W_{is}^v,W_{js}^{w}]|}{|K_3[W_{ij}^u,W_{is}^v,W_{js}^{w}]|}\geq 1-\delta^{1/100}.
$$
\end{claim}
\begin{proof}
Suppose towards a contradiction there is $W_{ij}^uW_{is}^vW_{js}^w\in \Omega_3$ such that for each $\sigma\in \{0,1\}$, $\frac{|R_{\sigma}\cap K_3[W_{ij}^u,W_{is}^v,W_{js}^{w}]|}{|K_3[W_{ij}^u,W_{is}^v,W_{js}^{w}]||}<1-\delta^{1/100}$.   To ease notation, let $A=W_{ij}^u$, $B=W_{is}^v$, and $C=W_{js}^w$.  

We now define a series of subsets of $A$ which will contain ``well behaved'' vertices.  First, we set $A_1=\{a\in A: a\sim_{js,vw} x_{ij}^u\}$.  Since $W_{ij}^uW_{is}^vW_{js}^w\in \Omega_3$,
$$
\delta^{1/4}|W_{ij}^u||W_{is}^v||W_{js}^w|\geq |Tr\cap K_3[W_{ij}^uW_{is}^vW_{js}^w]|\geq |A\setminus A_1|\delta^{1/8}|W_{is}^v||W_{js}^w|,
$$
Thus $|A\setminus A_1|\leq \delta^{-1/8}\delta^{1/4}|W_{ij}^u|= \delta^{1/8}|W_{ij}^u|$.  Now set 
$$
A_2=\{a\in A: |N_{R_2}(a)| \leq (\e_1')^{1/4} |B||C|\}.
$$
Because $W_{ij}^uW_{is}^vW_{js}^w\in \Omega_2$, we have that
$$
(\e_1')^{1/2}|A||B||C|\geq |R_2\cap K_3[A,B,C]|\geq |A\setminus A_2|(\e_1')^{1/4}|B||C|.
$$
Therefore, $|A\setminus A_2|\leq (\e_1')^{1/4}|A|$. Now set 
\begin{align*}
A_3&=\{a\in A: |N_{R_1}(a)|/|B||C|\in (\delta^{1/64},1-\delta^{1/64})\}\text{ and }\\
A_3'&=\{a\in A: |N_{R_1}(a)|/|B||C|\in (\delta^{1/128},1-\delta^{1/128})\}.
\end{align*}
We claim $x_{ij}^u\in A_3'$.  Suppose towards a contradiction that $x_{ij}^u\notin A'_3$.  Suppose first that $|N_{R_1}(x_{ij}^u)|\geq (1-\delta^{1/128})|B||C|$.  Then for all $a\in A_1$, since $a\sim_{js,vw}x_{ij}^u$, we have $|N_{R_1}(a)|\geq (1-\delta^{1/128}-\delta^{1/8})|B||C|$, and thus,
\begin{align*}
|R_1\cap K_3[W_{ij}^uW_{is}^vW_{js}^{w}]|\geq (1-\delta^{1/128}-\delta^{1/8})|A_1||B||C|&\geq (1-\delta^{1/128}-\delta^{1/8})(1-\delta^{1/8})|A||B||C|\\
&\geq (1-\delta^{1/100})|A||B||C|,
\end{align*}
a contradiction.  So we must have $|N_{R_1}(x_{ij}^u)|\leq \delta^{1/128} |B||C|$.  Then for all $a\in A_1\cap A_2$,  $a\sim_{js,vw}x_{ij}^u$ and $|N_{R_2}(a)|\leq (\e_1')^{1/4}|B||C|$ implies
$$
|N_{R_0}(a)|\geq (1-\delta^{1/128}-\delta^{1/8}-(\e_1')^{1/4})|B||C|.
$$
Therefore
\begin{align*}
|R_0\cap K_3[W_{ij}^uW_{is}^vW_{js}^{w}]|&\geq (1-\delta^{1/128}-\delta^{1/8}-(\e_1')^{1/4})|A_1\cap A_2||B||C|\\
&\geq (1-\delta^{1/128}-\delta^{1/8}-(\e_1')^{1/4})(1-\delta^{1/8}-(\e_1')^{1/4})|A||B||C|\\
&\geq (1-\delta^{1/100})|A||B||C|,
\end{align*}
again a contradiction.  Thus, we must have that $x_{ij}^u\in A'_3$.  This implies that for all $a\in A_1\cap A_2$, 
$$
|N_{R_1}(a)|\geq |N_{R_1}(x_{ij}^u)|-|N_{R_1}(x_{ij}^u)\Delta N_{R_1}(a)|\geq \delta^{1/128}|B||C|(1-\delta^{1/8})\geq \delta^{1/64}|B||C|, 
$$
and
\begin{align*}
|N_{R_0}(a)|\geq |N_{R_0}(x_{ij}^u)|- |N_{R_0}(a)|-|N_{R_0}(x_{ij}^u)\Delta N_{R_0}(a)|&\geq \delta^{1/128}|B||C|(1-\delta^{1/8}-(\e_1')^{1/4})\\
&\geq \delta^{1/64}|B||C|.
\end{align*}
Thus $a\in A_3$.  This shows that $A_1\cap A_2\subseteq A_3$, and therefore $|A_3|\geq |A|(1-\delta^{1/8}-(\e_1')^{1/4})$.  Now define

\begin{align*}
A_B&=\{a\in A: |\{b\in B: |N_{R_1}(ab)\Delta N_{R_1}(ax_{is}^v)|\leq \delta^{1/16}|C|\}|\geq (1-\delta^{1/16})|B|\}\text{ and }\\
A_C&=\{a\in A: |\{c\in C: |N_{R_1}(ac)\Delta N_{R_1}(ax_{js}^w)|\leq \delta^{1/16}|B|\}|\geq (1-\delta^{1/16})|C|\}.
\end{align*}
Observe that $4\delta^{1/4}|A||B||C|\geq |Tr\cap K_3[A,B,C]|\geq \delta^{1/16} |A\setminus A_B||B||C|$, and therefore $|A\setminus A_B|\leq \delta^{-1/16}4\delta^{1/4}|A|=4\delta^{3/16}|A|$.  A similar computation shows $|A\setminus A_C|\leq 4\delta^{3/16}|A|$.  Consequently, setting $A_4:=A_3\cap A_B\cap A_C$, we have that 
$$
|A_4|\geq |A_3|-|A\setminus A_C|-|A\setminus A_B|\geq |A|(1-8\delta^{3/16}- (\e_1')^{1/4}-\delta^{1/8})>0.
$$
Fix some $a_*\in A_4$.  We will use $a_*$ to control the other edges in the triple.  Let
\begin{align*}
S_1&=N_{R_1}(a^*), \text{ }S_0=N_{R_0}(a^*),\text{ and }S_{2}=N_{R_2}(a^*).
\end{align*}
Note $(B\cup C, S_0\cup S_1\cup S_2)$ is a $3$-partite edge colored $3$-graph. Since $a_*\in A_3$, $|S_1|/|B||C|\in (\delta^{1/64}, 1-\delta^{1/64})$.  Therefore, Corollary \ref{cor:twosticks} implies that one of the following hold.
\begin{enumerate}[(a)]
\item there is $B_1\subseteq B$ such that $|B_1|\geq \delta^{1/32}|B|/2$ and for all $b\in B_1$, $|N_{S_1}(b)|/|C|\in (\delta^{1/32}/2, 1-\delta^{1/32}/2)$.
\item  there is $C_1\subseteq C$ such that $|C_1|\geq \delta^{1/32}|C|/2$ and for all $c\in C_1$, $|N_{S_1}(c)|/|B|\in (\delta^{1/32}/2, 1-\delta^{1/32}/2)$.
\end{enumerate}
Without loss of generality, let us assume (a) holds (other case is symmetric).  Define $B_2=\{b\in B_1: |N_{S_2}(b)|\leq (\e_1')^{1/16}|C|\}$.  We claim $|B_2|\geq  \delta^{1/32}|B|/4$.  Indeed, we know that since $a_*\in A_2$, 
$$
(\e_1')^{1/4}|B||C|\geq |S_2|\geq (\e_1')^{1/16}|B_1\setminus B_2||C|.
$$
Thus, $|B_1\setminus B_2|\leq (\e_1')^{-1/16}(\e_1')^{1/4}|B|=(\e_1')^{1/12}|B|$, so 
$$
|B_2|\geq |B_1|-(\e_1')^{1/12}|B|\geq (\delta^{1/32}/2-(\e_1')^{1/12})|B|\geq \delta^{1/32}|B|/4.
$$
Note that for all $b\in B_2$, we have that $|N_{S_1}(b)|\geq \delta^{1/32}|C|/2\geq  \delta^{1/32}|C|/4$ and 
$$
|N_{S_0}(b)|\geq |C\setminus N_{S_1}(b)|-|N_{S_2}(b)|\geq (\delta^{1/32}/2-(\e_1')^{1/16})|C|\geq \delta^{1/32}|C|/4.
$$

Now, let 
$$
B_3=\{b\in B_2: |N_{S_1}(b)\Delta N_{S_1}(x_{is}^v)|\leq \delta^{1/16}|C|\}.
$$
Since $a_*\in A_B$, 
$$
|B_3|\geq |B_2|-\delta^{1/16}|B|\geq  (\delta^{1/32}/4-\delta^{1/16})|B|\geq \delta^{1/32}|B|/8>0.
$$
  Fix some $b_*\in B_3$ and set $Q_0=N_{S_0}(b_*)$ and $Q_1=N_{S_1}(b_*)$.  By above, since $b_*\in B_2$, $\min\{|Q_1|,|Q_0|\}\geq \delta^{1/32}|C|/4$.

We claim $|S_1\cap K_2[B_3, Q_1]|\geq (1-10\delta^{1/32})|Q_1||B_3|$.  Indeed, fix $b\in B_3$.  Then we know that  $|N_{S_1}(b)\Delta N_{S_1}(x_{is}^v)|\leq \delta^{1/16}|C|$ and $|N_{S_1}(b_*)\Delta N_{S_1}(x_{is}^v)|\leq \delta^{1/16} |C|$, and therefore $|N_{S_1}(b)\Delta N_{S_1}(b_*)|\leq 2\delta^{1/16} |C|$.  Consequently,
\begin{align*}
|N_{S_1}(b)\cap Q_1|\geq |Q_1|-2\delta^{1/16} |C|\geq |Q_1|(1-2\delta^{1/16} (|C|/|Q_1|)&\geq |Q_1|(1-2\delta^{1/16} (4\delta^{-1/32}))\\
&\geq |Q_1|(1-10\delta ^{1/32}).
\end{align*}  
This shows that $|S_1\cap K_2[B_3, Q_1]|\geq (1-10\delta^{1/32})|Q_1||B_3|$.

Similarly, we claim $|S_0\cap K_2[B_3,Q_0]|\geq (1-10\delta^{1/32})|B_3||Q_0|$.  Indeed, for all $b\in B_3$, we have $|N_{S_2}(b)|\leq (\e_1')^{1/16}|C|$ and, as above, $|N_{S_1}(b)\Delta N_{S_1}(b_*)|\leq 2\delta^{1/16} |C|$.  Thus $|N_{S_0}(b)\Delta N_{S_0}(b_*)|\leq ((\e_1')^{1/16} +2\delta^{1/16})|C|$.  Therefore
\begin{align*}
|N_{S_0}(b)\cap Q_0|\geq |Q_0|-((\e_1')^{1/4}+2\delta^{1/16}) |C|&\geq |Q_0|(1-((\e_1')^{1/4}+2\delta^{1/16})  (|C|/|Q_0|)\\
&\geq |Q_0|(1-((\e_1')^{1/4}+2\delta^{1/16}) 4\delta^{-1/32} )\\
&\geq |Q_1|(1-10\delta^{1/32}),
\end{align*} 
where the last inequality uses the definition of $\e_1'$.  This shows that  $|S_0\cap K_2[B_3,Q_0]|\geq (1-10\delta^{1/32})|B_3||Q_0|$.  Now let 
\begin{align*}
Q_1'&=\{c\in Q_1: |N_{S_1}(c) \cap B_3|\geq (1-\sqrt{10}\delta^{1/64})|B_3|\}\text{ and }\\
Q_0'&=\{c\in Q_0: |N_{S_0}(c) \cap B_3|\geq (1-\sqrt{10}\delta^{1/64})|B_3|\}.
\end{align*}
Since both $|S_1\cap K_2[B_3, Q_1]|\geq (1-10\delta^{1/32})|Q_1||B_3|$ and $|S_0\cap K_2[B_3,Q_0]|\geq (1-10\delta^{1/32})|B_3||Q_0|$, we have that $|Q_1'|\geq (1-\sqrt{10}\delta^{1/64})|Q_1|$ and $|Q_0'|\geq (1-\sqrt{10}\delta^{1/64})|Q_0|$. Finally, let 
$$
C^*=\{c\in C: |N_{S_1}(c)\Delta N_{S_1}(x_{js}^w)|\leq \delta^{1/16}|B|\}.
$$
Since $a_*\in A_C$, $|C^*|\geq (1-\delta^{1/16})|C|$.  Thus, 
$$
|Q_1'\cap C^*|\geq (1-\sqrt{10}\delta^{1/64})|Q_1|-\delta^{1/16}|C|\geq ((1-\sqrt{10}\delta^{1/64})\delta^{1/32}/4- \delta^{1/16}|C|\geq \delta^{1/32}|C|/10.
$$
Similarly, 
$$
|Q_0'\cap C^*|\geq (1-\sqrt{10}\delta^{1/64})|Q_1|-\delta^{1/16}|C|\geq ((1-\sqrt{10}\delta^{1/64})\delta^{1/32}/4- \delta^{1/16}|C|\geq \delta^{1/32}|C|/10.
$$
Consequently, there are $c_1\in Q'_1\cap C^*$ and $c_0\in Q'_0\cap C^*$.  Since $c_0,c_1\in C^*$, we can see that $|N_{S_1}(c_1)\Delta N_{S_1}(c_0)|\leq 2\delta^{1/16}|B|$. However, we also have that 
$$
|N_{R_1}(c_1)\cap N_{R_0}(c_0)\cap B_3|\geq (1-2\sqrt{10}\delta^{1/64})|B_3|\geq (1-2\sqrt{10}\delta^{1/64})\delta^{1/32}|B|/8>2\delta^{1/16}|B|.
$$
But this is a contradiction, since $N_{S_1}(c)\cap N_{S_0}(c_0)\cap B_3\subseteq  N_{S_1}(c_1)\Delta N_{S_1}(c_0)$.
\end{proof}

Let $\ell_1=\lceil \delta^{-4}m^4\rceil$.   Suppose $V_iV_j\notin \Psi$ and $1\leq u\leq \ell_{ij}$ is such that $W_{ij}^u$ is nontrivial.  Define $\mathbf{W}_{ij}^u=\bigcup_{P_{ij}^{\alpha}\in W_{ij}^u}P_{ij}^u$, let $\mathbf{G}_{ij}^u$ be the bipartite graph $(V_i\cup V_j, \mathbf{W}_{ij}^u)$, and define $\rho_{ij}(u)=\frac{|\mathbf{W}_{ij}^u|}{|V_i||V_j|}$.  By Fact \ref{fact:adding}, $\textbf{G}_{ij}^u$ has $\dev_2(\ell(\e_2''(\ell))^{1/4})$ and $|\mathbf{W}_{ij}^u|=(1\pm \ell(\e_2''(\ell))^{1/4})\frac{|W_{ij}^u||V_i||V_j|}{\ell}$.  Using the size estimate above and the fact that $W_{ij}^u$ is non-trivial, we have
$$
\rho_{ij}(u)=(1\pm \ell(\e_2''(\ell))^{1/4})|W_{ij}^u|/\ell\geq (1\pm \ell(\e_2''(\ell))^{1/4})\delta^{1/2}/\ell\geq 2\ell(\e_2''(\ell))^{1/4},
$$
where the last inequality is by our choice of $\e_2''(\ell)$.   Now set $p_{ij}(u)=\rho_{ij}(u)^{-1}/\ell_1$, and let $s_{ij}(u)=[1/p_{ij}(u)]$.  Observe that $\rho_{ij}p_{ij}=\frac{1}{\ell_1}$.  Note $(\e_2''(\ell))^{1/4}\geq 10(1/s|V_i|)^{1/5}$ (since $n$ is very large), and since $W_{ij}^u$ is non-trivial and $\ell(\e_2''(\ell))^{1/4}<1/4$, 
$$
\rho_{ij}(u)\geq (1\pm \ell(\e_2''(\ell))^{1/4})|W_{ij}^u|/\ell\geq \delta^{1/2} (\ell/m_{ij})/\ell=\delta^{1/2}/m_{ij}\geq \delta^{1/2}/m.
$$
Further, $0<p_{ij}(u)<\rho_{ij}(u)/2$ since 
$$
p_{ij}(u)\leq (1\pm \ell (\e_2''(\ell))^{1/4})^{-1}m\delta^{-1/2}\ell_1^{-1}\leq m\delta^{-1/2}\delta^{4}/m^4\leq\delta^{3/2}/m^3< \rho_{ij}(u)/2,
$$
where the last inequality uses that $\rho_{ij}(u)\geq \delta^{1/2}/m$.  Thus by Lemma \ref{lem:3.8}, there is a partition 
$$
\mathbf{W}_{ij}^u=\mathbf{W}_{ij}^u(0)\cup \ldots \cup \mathbf{W}_{ij}^u(s_{ij}(u)),
$$
so that $|\mathbf{W}_{ij}^u(0)|\leq \rho_{ij}p_{ij}(1+\e_1')|V_i||V_j|$ and for each $1\leq x\leq s_{ij}(u)$, the bipartite graph $\mathbf{G}_{ij}^u(x):=(V_i\cup V_j, \mathbf{W}_{ij}^u(x))$ has $\dev_2(\ell(\e_2''(\ell))^{1/4}, \rho_{ij}p_{ij})$, i.e. $\dev_2(\ell(\e_2''(\ell))^{1/4}, 1/\ell_1)$.  Since $(\e_2''(\ell))^{1/4}m<\e_2(\ell_1)$, and by definition of $\e_2''$, we have that for each $1\leq x\leq s_{ij}(u)$, $\mathbf{G}_{ij}^u(x)$ has $\dev_2(\e_2(\ell_1), 1/\ell)$.  Let $s_{ij}=\sum_{1\leq u\leq m_{ij}}s_{ij}(u)$. Give a re-enumeration 
$$
\{X_{ij}^1,\ldots, X_{ij}^{s_{ij}}\}=\{\mathbf{W}_{ij}^u(v): 1\leq v\leq s_{ij}(u), 1\leq u\leq m_{ij}\}.
$$
Then let $X_{ij}^{s_{ij}+1},\ldots, X_{ij}^{\ell_1}$ be any partition of $K_2[V_i,V_j]\setminus \bigcup_{x=1}^{s_{ij}}X_{ij}^x$.  

For $V_iV_j\in \Psi$ choose a partition $K_2[V_i,V_j]=X_{ij}^1\cup \ldots \cup X_{ij}^{\ell_1}$ such that for each $1\leq x\leq \ell_1$, $X_{ij}^{\ell_1}$ has $\dev_2( \e_2 (\ell_1), 1/\ell_1)$ (such a partition exists by Lemma \ref{lem:3.8}).  Now define $\calQ$ to be the decomposition of $V$ with $\calQ_1=\{V_i: i\in [t]\}$ and $\calQ_2=\{X_{ij}^v: v\leq \ell_1, ij\in {[t]\choose 2}\}$.  We claim this is a $(t,\ell_1, \e_1,\e_2(\ell_1))$-decomposition of $V$. Indeed, by construction, any $xy\in {V\choose 2}$ which is not in an element of $\calQ_2$ satisfying $\disc_2(\e_2(\ell_1),1/\ell_1)$ is in the set
$$
\Gamma:=\bigcup_{V_iV_j\notin \Psi}X_{ij}^{s_{ij}+1}\cup \ldots \cup X_{ij}^{\ell_1}.
$$
Observe that 
\begin{align}\label{align:3}
|\Gamma|\leq \sum_{V_iV_j\notin \Psi}\sum_{u=1}^{m_{ij}}|\mathbf{W}_{ij}^u(0)|+|K_2[V_i,V_j]\setminus (\bigcup_{P_{ij}^{\alpha}\in W_{ij}}P_{ij}^{\alpha})|.
\end{align}
We have that
\begin{align*}
\sum_{V_iV_j\notin \Psi}\sum_{u=1}^{m_{ij}}|\mathbf{W}_{ij}^u(0)|\leq \sum_{V_iV_j\notin \Psi}m_{ij}(1+\e_1')\rho_{ij}p_{ij}|V_i||V_j|&\leq {t\choose 2}m(1+2\e_1')(n/t)^2/\ell_1\\
&=\delta^4 {t\choose 2}(1+2\e_1')(n/t)^2/m^3\\
&\leq 2\delta^2{n\choose 2}/m^3,
\end{align*}
where the last inequality is because $n$ is large.  Then, by definition of $\delta$ and $m$, this shows that $\sum_{V_iV_j\notin \Psi}\sum_{u=1}^{m_{ij}}|\mathbf{W}_{ij}^u(0)|\leq \e_1 {n\choose 2}/2$.  We also have that 
\begin{align*}
\sum_{V_iV_j\notin \Psi} |K_2[V_i,V_j]\setminus (\bigcup_{P_{ij}^{\alpha}\in W_{ij}}P_{ij}^{\alpha})|&\leq \sum_{V_iV_j\notin \Psi}(|V_i||V_j|-|W_{ij}|(1+\e_2'(\ell))\frac{|V_i||V_j|}{\ell})\\
&=\sum_{V_iV_j\notin \Psi} |V_i||V_j|(1-\ell_{ij}(1+\e_2'(\ell))/\ell)\\
&\leq \sum_{V_iV_j\notin \Psi}|V_i||V_j|(1-(1-2(\e_1')^{3/4})(1+\e_2'(\ell)))\\
&\leq \sum_{V_iV_j\notin \Psi}|V_i||V_j|(\e_1')^{1/8}\\
&\leq (\e_1')^{1/8}{n\choose 2}.
\end{align*}
Combining these with (\ref{align:3}) yields that $|\Gamma|\leq \e_1{n\choose 2}/2+(\e_1')^{1/8}{n\choose 2}\leq \e_1{n\choose 2}$, and therefore, $\calQ$ is a $(t,\ell, \e_1,\e_2(\ell))$-decomposition of $V$.

We now show that $\calQ$ is $\e_1/6$-homogeneous with respect to $H$.  We will first show that for any $W_{ij}^uW_{is}^vW_{js}^{w}\in \Omega_3$, $\mathbf{G}_{ijs}^{uvw}:=(V_i\cup V_j\cup V_s, \mathbf{W}_{ij}^{u}\cup \mathbf{W}_{is}^v\cup \mathbf{W}_{js}^w)$ is $2\delta^{1/100}$-homogeneous with respect to $H$, and second,  almost all $xyz\in K_3^{(2)}(\mathbf{G}_{ijs}^{uvw})$ are in an $\e_1/6$-homogenous triad of $\calQ$.

Fix $W_{ij}^uW_{is}^vW_{js}^{w}\in \Omega_3$.  We know by Claim \ref{cl:hom}, that there is $\sigma\in \{0,1\}$ such that 
$$
|R_{\sigma}\cap K_3[W_{ij}^u,W_{is}^v,W_{js}^{w}]|\geq (1-\delta^{1/100})|K_3[W_{ij}^u,W_{is}^v,W_{js}^{w}]|.
$$
This implies, by (\ref{align:1}), and definition of $R_{\sigma}$ that the following holds, where $E^1=E$ and $E^0={V\choose 3}\setminus E^1$ (recall $E=E(H)$).
\begin{align*}
|E^{\sigma}\cap K_3^{(2)}(\mathbf{G}_{ijs}^{uvw})|&\geq (1-\delta^{1/100})(1-\e_1'')|K_3[W_{ij}^u,W_{is}^v,W_{js}^{w}]|(1-\ell^3\e_2'(\ell))|V_i||V_j||V_s|\frac{1}{\ell^3}\\
&=(1-\delta^{1/100})(1-\e_1'')(1-\ell^3\e_2'(\ell))\frac{|W_{ij}^u||W_{is}^v||W_{js}^w|}{\ell^3}.
\end{align*}
On the other hand, note that by (\ref{align:1}),
\begin{align*}
|K_3^{(2)}(\mathbf{G}_{ijs}^{uvw})|=|W_{ij}^u||W_{is}^v||W_{js}^w|(1\pm \ell^3\e_2'(\ell))\frac{|V_i||V_j||V_s|}{\ell^3}.
\end{align*}
Combining this with the above, we see that 
\begin{align*}
|E^{\sigma}\cap K_3^{(2)}(\mathbf{G}_{ijs}^{uvw})|&\geq (1-\delta^{1/100})(1-\e_1'')(1-\ell^3\e_2'(\ell))(1+\ell^3\e_2'(\ell))^{-1}|K_3^{(2)}(\mathbf{G}_{ijs}^{uvw})|\\
&\geq (1-2\delta^{1/100})|K_3^{(2)}(\mathbf{G}_{ijs}^{uvw})|,
\end{align*}
where the last inequality is by definition of $\e_2'$ and $\e_1''$.  This shows $\mathbf{G}_{ijs}^{uvw}$ is $2\delta^{1/100}$-homogeneous.  We now show that almost all $xyz\in K_3^{(2)}(\mathbf{G}_{ijs}^{uvw})$ are in an $\e_1/6$-homogeneous triad of $\calQ$.  Set 
$$
\Sigma_0(ijs,uvw)=\{0,\ldots, s_{ij}(u)\}\times \{0,\ldots, s_{is}(v)\}\times \{0,\ldots, s_{js}(w)\}.
$$
Given $(x,y,z)\in \Sigma_0$, set 
$$
\mathbf{G}_{ijs}^{uvw}(x,y,z)=(V_i\cup V_j\cup V_s; \mathbf{W}_{ij}^u(x)\cup  \mathbf{W}_{is}^v(y)\cup  \mathbf{W}_{js}^w(z)).
$$
Note that $K_3^{(2)}(\mathbf{G}_{ijs}^{uvw})=\bigcup_{(x,y,z)\in \Sigma_0(ijs,uvw)}K_3^{(2)}(\mathbf{G}_{ijs}^{uvw}(x,y,z))$.  Define 
$$
\Sigma_1(ijs,uvw)=\{ (x,y,z)\in \{0,\ldots, s_{ij}(u)\}\times \{0,\ldots, s_{is}(v)\}\times \{0,\ldots, s_{js}(w)\}: x,y\text{ or }z\text{ is }0\},
$$
 and set $\Sigma_2(ijs,uvw)=\Sigma_0(ijs,uvw)\setminus \Sigma_1(ijs,uvw)$. Note that by construction, for all $(x,y,z)\in \Sigma_2(ijs,uvw)$, $\mathbf{G}_{ijs}^{uvw}(x,y,z)\in \triads(\calQ)$.  Oberve that
\begin{align*}
\sum_{(x,y,z)\in \Sigma_1(ijs,uvw)}|K_3^{(2)}(\mathbf{G}_{ijs}^{uvw}(x,y,z)|&\leq |\mathbf{W}_{ij}^u(0)||V_s|+|\mathbf{W}_{is}^v(0)||V_j|+|\mathbf{W}_{js}^w(0)||V_i|\\
&\leq (1+\e_1')|V_i||V_j||V_s|(\rho_{ij}p_{ij}+\rho_{is}p_{is}+\rho_{js}p_{js})\\
&= 3(1+\e_1')|V_i||V_j||V_s|/\ell_1\\
&\leq 3(1+\e_1')\delta^4|V_i||V_j||V_s|m^{-4}\\
&\leq 3(1+\e_1')\delta^4(|W_{ij}^u||W_{is}^v||W_{js}^w|/\ell^3)^{-1}m^{-4}|K_3^{(2)}(\mathbf{G}_{ijs}^{uvw})|\\
&\leq 3(1+\e_1')\delta^4(\delta^{1/2}/m)^{-3}m^{-4}|K_3^{(2)}(\mathbf{G}_{ijs}^{uvw})|\\
&=3(1+\e_1')\delta^{1/2}m^{-1}|K_3^{(2)}(\mathbf{G}_{ijs}^{uvw})|\\
&<\delta|K_3^{(2)}(\mathbf{G}_{ijs}^{uvw})|,
\end{align*}
where the last inequality uses the definition of $m$. Let $\Sigma_3(ijs,uvw)$ be the set of $(x,y,z)\in \Sigma_2(ijs,uvw)$ such that
$$
 |E^{\sigma}\cap K_3^{(2)}(\mathbf{G}_{ijs}^{uvw}(x,y,z))|<(1-\delta^{1/200})|K_3^{(2)}(\mathbf{G}_{ijs}^{uvw}(x,y,z))|,
$$
 and set $\Sigma_4(ijs,uvw)=\Sigma_2(ijs,uvw)\setminus \Sigma_3(ijs,uvw)$.  By definition, and since $\delta^{1/200}<\e_1/6$, every triad of the form $K_3^{(2)}(\mathbf{G}_{ijs}^{uvw}(x,y,z))$ for $(x,y,z)\in \Sigma_4(ijs,uvw)$ is in a $\e_1/6$-homogeneous triad of $\calQ$.  We now show that $\bigcup_{(x,y,z)\in \Sigma_4(ijs,uvw)}K_3^{(2)}(\mathbf{G}_{ijs}^{uvw})$ is most of $K_3^{(2)}(\mathbf{G}_{ijs}^{uvw})$.  Observe
\begin{align*}
|E^{\sigma}\cap K_3^{(2)}(\mathbf{G}_{ijs}^{uvw}))|&\leq \sum_{(x,y,z)\in \Sigma_1(ijs,uvw)}|K_3^{(2)}(\mathbf{G}_{ijs}^{uvw}(x,y,z))|\\
&+(1-\delta^{1/200})\sum_{(x,y,z)\in \Sigma_3(ijs,uvw)}|K_3^{(2)}(\mathbf{G}_{ijs}^{uvw}(x,y,z))|\\
&+\sum_{(x,y,z)\in \Sigma_4(ijs,uvw)}|K_3^{(2)}(\mathbf{G}_{ijs}^{uvw}(x,y,z))|\\
&\leq \delta|K_3^{(2)}(\mathbf{G}_{ijs}^{uvw})|+(1-\delta^{1/200})\sum_{(x,y,z)\in \Sigma_3(ijs,uvw)}|K_3^{(2)}(\mathbf{G}_{ijs}^{uvw}(x,y,z))|+\\
&\sum_{(x,y,z)\in \Sigma_4(ijs,uvw)}|K_3^{(2)}(\mathbf{G}_{ijs}^{uvw}(x,y,z))|.
\end{align*}
Thus since $|E^{\sigma}\cap K_3^{(2)}(\mathbf{G}_{ijs}^{uvw})|\geq (1-2\delta^{1/100})|K_3^{(2)}(\mathbf{G}_{ijs}^{uvw})|$,
\begin{align*}
(1-2\delta^{1/100}-\delta)|K_3^{(2)}(\mathbf{G}_{ijs}^{uvw}))|&\leq (1-\delta^{1/200})\sum_{(x,y,z)\in \Sigma_3(ijs,uvw)}|K_3^{(2)}(\mathbf{G}_{ijs}^{uvw}(x,y,z))|\\
&+\sum_{(x,y,z)\in \Sigma_4(ijs,uvw)}|K_3^{(2)}(\mathbf{G}_{ijs}^{uvw}(x,y,z))|\\
&=\sum_{(x,y,z)\in \Sigma_2(ijs,uvw)}|K_3^{(2)}(\mathbf{G}_{ijs}^{uvw}(x,y,z))|\\
&-\delta^{1/200}\sum_{(x,y,z)\in \Sigma_3(ijs,uvw)}|K_3^{(2)}(\mathbf{G}_{ijs}^{uvw}(x,y,z))|.
\end{align*}
Rearrangine this, we have the following upper bound for $\sum_{(x,y,z)\in \Sigma_3(ijs,uvw)}|K_3^{(2)}(\mathbf{G}_{ijs}^{uvw}(x,y,z))|$.
\begin{align*}
& \delta^{-1/200}\Big(\sum_{(x,y,z)\in \Sigma_2(ijs,uvw)}|K_3^{(2)}(\mathbf{G}_{ijs}^{uvw}(x,y,z))|-(1-2\delta^{1/100}-\delta)|K_3^{(2)}(\mathbf{G}_{ijs}^{uvw}))|\Big)\\
&\leq \delta^{-1/200}|K_3^{(2)}(\mathbf{G}_{ijs}^{uvw}))|3\delta^{1/100}\\
&\leq 3\delta^{1/200}|K_3^{(2)}(\mathbf{G}_{ijs}^{uvw}))|.
\end{align*}
Consequently,
\begin{align*}
\sum_{(x,y,z)\in \Sigma_4(ijs,uvw)}|K_3^{(2)}(\mathbf{G}_{ijs}^{uvw}(x,y,z))|&\geq |K_3^{(2)}(\mathbf{G}_{ijs}^{uvw})|(1-3\delta^{1/200}).
\end{align*}

We have now shown that $\bigcup_{(x,y,z)\in \Sigma_4(ijs,uvw)}K_3^{(2)}(\mathbf{G}_{ijs}^{uvw})$ is most of $K_3^{(2)}(\mathbf{G}_{ijs}^{uvw})$, and for all $(x,y,z)\in \Sigma_4(ijs,uvw)$, $\mathbf{G}_{ijs}^{uvw}(x,y,z)$ is an $\e_1/6$-homogeneous triad of $\calQ$.  For all $(x,y,z)\in \Sigma_4(ijs,uvw)$, $W_{ij}^u(x),W_{is}^v(y),W_{js}^w(z)$ all have $\dev_2(\e_2(\ell_1),1/\ell_1)$, and thus, by Proposition \ref{prop:homimpliesrandome},  $\mathbf{G}_{ijs}^{uvw}(x,y,z)$ has $\dev_{2,3}(\e_1,\e_2(\ell_1))$ with respect to $H$.

Using this and our lower bound on the size of $Y_3$, we can now give the following lower bound on the number of triples $xyz\in {V\choose 3}$ in a $\dev_{2,3}(\e_2(\ell_1),\e_1)$-regular triad of $\calP$. 
\begin{align*}
\sum_{W_{ij}^uW_{is}^vW_{js}^w\in \Omega_3}\sum_{(x,y,z)\in \Sigma_4(ijs,uvw)}|K_3^{(2)}(\mathbf{G}_{ijs}^{uvw}(x,y,z))|&\geq \sum_{W_{ij}^uW_{is}^vW_{js}^w\in \Omega_3}(1-3\delta^{1/200})|K_3^{(2)}(\mathbf{G}_{ijs}^{uvw})|\\
&=(1-3\delta^{1/200})\sum_{W_{ij}^uW_{is}^vW_{js}^w\in \Omega_3}|K_3^{(2)}(\mathbf{G}_{ijs}^{uvw})|\\
&\geq (1-3\delta^{1/200})(1-3\delta^{1/2}){n\choose 3}\\
&\geq (1-\e_1){n\choose 3},
\end{align*}
where the last inequality is by definition of $\delta$.  This finishes the proof.

\end{proofof}

\vspace{2mm}

\appendix

\section{Proof of Proposition \ref{prop:homimpliesrandome}}\label{app:counting}

We will use the following fact.

\begin{lemma}\label{lem:count}
For all $\delta,r,\mu \in (0,1]$ satisfying $2^{12}\delta<\mu^2r^{12}$, the following holds.   Suppose $G=(V_1\cup V_2\cup V_3, E)$ is a $3$-partite graph such that for each $ij\in {[3]\choose 2}$, $||V_i|-|V_j||\leq \delta |V_i|$ and $G[V_i,V_j]$ has $\dev_2(\delta,r)$.  Given $u_0v_0w_0\in K_3^{(2)}(G)$, define 
$$
K_{2,2,2}[u_0,v_0,w_0]=\{u_1v_1w_1\in K_3[V_1,V_2,V_3] : \text{for each }\e\in \{0,1\}^3, (u_{\e_1},v_{\e_2},w_{\e_3})\in K_3^{(2)}(G)\}.
$$
Then if $J:=\{uvw\in K_3^{(2)}(G):  |K_{2,2,2}[u_0,v_0,w_0]|\leq (1+\mu)r^9|V_1||V_2||V_3|\}$,  we have that $|J|\geq (1-\mu)r^3|V_1||V_2||V_3|$.
\end{lemma}

\begin{proof}
Let $K^G_{2,2,2}[V_1,V_2,V_3]$ be the set
$$
\{(u_0,u_1,w_0,w_1,z_0,z_1)\in V_1^2\times V_2^2\times V^3_3: \text{ for each }\e\in \{0,1\}^3, u_{\e_1}w_{\e_2}z_{\e_3}\in R\cap K_3^{(2)}(G)\}.
$$
By Theorem 3.5 in \cite{Gowers.20063gk},
$$
|K_{2,2,2}[V_1,V_2,V_3]|\leq r^{12}|V_1|^2|V_2|^2|V_3|^2+2^{12}\delta^{1/4}|V_1|^2|V_2|^2|V_3|^2.
$$
Suppose towards a contradiction that $|J|>(1-\mu)r^3|V_1||V_2||V_3|$.  Then 
$$
 |K_{2,2,2}[V_1,V_2,V_3]|\geq |J|(1+\mu)r^9|V_1||V_2||V_3|>(1-\mu^2)r^{12}|V_1|^2|V_2|^2|V_3|^2,
$$
Combining with the above, this implies $r^{12}+2^{12}\delta^{1/4}>(1-\mu^2)r^{12}$, which implies $\mu^2r^{12}<2^{12}\delta^{1/4}$, a contradiction.
\end{proof}

\noindent {\bf Proof of Proposition \ref{prop:homimpliesrandome}.} Fix $0<\e<1/2$, $0<d_2<1/2$, and $0<\delta\leq (d_2/2)^{48}$, and choose $N$ sufficiently large.

Suppose $H=(V_1\cup V_2\cup V_3, R)$ is a $3$-partite $3$-graph on $n\geq N$ vertices and for each $i,j\in [3]$, $||V_i|-|V_j||\leq \delta |V_i|$.  Suppose $G=(V_1\cup V_2\cup V_3,E)$ is a $3$-partite graph, where for each $1\leq i<j\leq 3$, $G[V_i,V_j]$ has $\dev_2(\delta,d_2)$, and assume $|R\cap K_3^{(2)}(G)|\leq \e |K_3^{(2)}(G)|$.  Let $d$ be such that $|R\cap K_3^{(2)}(G)|=d|K_3^{(2)}(G)|$.  By assumption $d\leq \e$.  Let $g(x,y,z):{V\choose 3}\rightarrow [0,1]$ be equal to $1-d$ if $xyz\in R\cap K_3^{(2)}(G)$, $-d$ if $xyz\in K_3^{(2)}(G)\setminus R$ and $0$ otherwise. Given $u_0v_0w_0\in K_3^{(2)}(G)$, define 
$$
K_{2,2,2}[u_0,v_0,w_0]=\{u_1v_1w_1\in K_3[V_1,V_2,V_3] : \text{for each }(i,j,k)\in \{0,1\}^3, (u_{i},v_{j},w_{k})\in K_3^{(2)}(G)\}.
$$
Let $\mu=d_2^{12}$. Note that $2^{12}\delta<(d_2/2)^{36}<d_2^{36}=\mu^2 d_2^{12}$.  Set 
$$
J:=\{uvw\in K_3^{(2)}(G):  |K_{2,2,2}[u_0,v_0,w_0]|\leq (1+\mu)d_2^9|V_1||V_2||V_3|\}.
$$
By Lemma \ref{lem:count}, we have that $|J|\geq (1-\mu)d_2^3|V_1||V_2||V_3|$.  Now set
$$
I_1=\{(u_0,u_1,w_0,w_1,z_0,z_1)\in V_1^2\times V_2^2\times V^3_3: \text{ for each }(i,j,k)\in \{0,1\}^3, u_iw_jz_k\in R\cap K_3^{(2)}(G)\}
$$
 and let 
 $$
 I_2=\{(u_0,u_1,w_0,w_1,z_0,z_1)\in (V_1^2\times V_2^2\times V^3_3)\setminus I_1: \text{ for each }(i,j,k)\in \{0,1\}^3, u_iw_jz_k\in K_3^{(2)}(G)\}.
 $$

Then 
\begin{align*}
\sum_{u_0,u_1\in V_1}\sum_{w_0,w_1\in V_2}\sum_{z_0,z_1\in V_3}\prod_{(i,j,k)\in \{0,1\}^3}g(u_i,w_j,z_k)&\leq |\sum_{u_0,u_1\in V_1}\sum_{w_0,w_1\in V_2}\sum_{z_0,z_1\in V_3}\prod_{(i,j,k)\in \{0,1\}^3}g(u_i,w_j,z_k)|\\
&\leq \sum_{u_0,u_1\in V_1}\sum_{w_0,w_1\in V_2}\sum_{z_0,z_1\in V_3}|\prod_{(i,j,k)\in \{0,1\}^3}g(u_i,w_j,z_k)|\\
&=\sum_{(u_0,u_1,w_0,w_1,z_0,z_1)\in I_1}(1-d)^9 \\
&+ \sum_{(u_0,u_1,w_0,w_1,z_0,z_1)\in I_2}|\prod_{(i,j,k)\in \{0,1\}^3}g(u_i,w_j,z_k)|.
\end{align*}

For each $(u_0,u_1,w_0,w_1,z_0,z_1)\in I_2$,  $|\prod_{(i,j,k)\in \{0,1\}^3}g(u_i,w_j,z_k)|\leq d(1-d)^8$, since at least one of the $g(u_i,w_j,z_k)$ is equal to $-d$, and $|-d|<|1-d|$ (since $d\leq \e<1/2$).  Thus we have by above, that
$$
\sum_{u_0,u_1\in V_1}\sum_{w_0,w_1\in V_2}\sum_{z_0,z_1\in V_3}\prod_{(i,j,k)\in \{0,1\}^3}g(u_i,w_j,z_k)\leq (1-d)^9|I_1|+d(1-d)^8|I_2|
$$
Note 
\begin{align*}
|I_1|&\leq \sum_{u_0w_0z_0\in J}|K_{2,2,2}(u_0,w_0,z_0)|+\sum_{u_0w_0z_0\in R\setminus J}|K_{2,2,2}(u_0,w_0,z_0)|\\
&\leq |J|(1+\mu)d_2^9|V_1||V_2||V_3|+|R\setminus J||R|\\
&\leq |R|(1+\mu)d_2^9|V_1||V_2||V_3|+\mu d_2^3|V_1||V_2||V_3|d|K_3^{(2)}(G)|\\
&\leq d|K_3^{(2)}(G)|(1+\mu)d_2^9|V_1||V_2||V_3|+\mu d_2^3|V_1||V_2||V_3|d|K_3^{(2)}(G)|\\
&\leq |V_1||V_2||V_3||K_3^{(2)}(G)|(d(1+\mu)d_2^{12} +  d d_2^{12}),
\end{align*}
where the last inequality is by definition of $\mu$.  By the counting lemma (Theorem 3.5 in \cite{Gowers.20063gk}), $|K_3^{(2)}(G)|\leq (1+2^3\delta^{1/4})|V_1||V_2||V_3|$.  Therefore, we have that 
$$
|I_1|\leq |V_1|^2|V_2|^2|V_3|^2(1+2^3\delta^{1/4})(d(1+\mu)d_2^{12} +  d d_2^{12})\leq 3d d_2^{12}|V_1|^2|V_2|^2|V_3|^2.
$$
On the other hand, $|I_2|\leq |K_{2,2,2}[V_2,V_2,V_3]|$, which, by Theorem 3.5 in \cite{Gowers.20063gk}, has size at most $(d_2^{12}+2^{12}\delta^{1/4})|V_1|^2|V_2|^2|V_3|^2$.  Combining the bounds above with the fact that $d\leq \e$, we have that 
\begin{align*}
\sum_{u_0,u_1\in V_1}\sum_{w_0,w_1\in V_2}\sum_{z_0,z_1\in V_3}\prod_{(i,j,k)\in \{0,1\}^3}g(u_i,w_j,z_k)&\leq (1-d)^9|I_1|+d(1-d)^8|I_2|\\
&\leq |V_1|^2|V_2|^2|V_3|^2(3\e d_2^{12}+ \e (d_2^{12}+2^{12}\delta^{1/4}))\\
&\leq 6\e d_2^{12}|V_1|^2|V_2|^2|V_3|^2,
\end{align*}
where the last inequality is since $\delta<(d_2/2)^{48}$.  This shows $(H,G)$ has $\dev_{2,3}( \delta,6\e)$.

\end{document}